\pdfoutput=1
\documentclass{amsart}
\usepackage[utf8]{inputenc}
\usepackage[english]{babel}
\usepackage{enumitem}
\usepackage[x11names]{xcolor}
\usepackage[foot]{amsaddr}
\usepackage{booktabs}
\usepackage{tikz}
\usepackage{tikz-cd}
\usepackage{amsmath,amssymb}
\usepackage{amsthm,thmtools}
\usepackage{float}
\usepackage{lmodern}
\usepackage{microtype}
\usepackage{mathtools}
\usepackage{mathrsfs}
\usepackage{url}
\usepackage[colorlinks=true,linkcolor=red!80!black,citecolor=blue]{hyperref} 
\usetikzlibrary{arrows,decorations.markings}

\tikzset{>=stealth',vertex/.style={circle,fill,inner sep=1.3pt}}

\definecolor{codegreen}{rgb}{0,0.8,0}

\addto\extrasenglish{%
}

\newcommand{\myautoref}[2]{\hyperref[#2]{\autoref*{#1}\ref*{#2}}} 

\theoremstyle{plain}
\newtheorem{thm}{Theorem}[section]
\newtheorem{prop}[thm]{Proposition}
\newtheorem{lem}[thm]{Lemma}

\theoremstyle{definition}
\newtheorem{definition}[thm]{Definition}
\newtheorem*{note*}{Remark}

\newcommand{\C}{\mathbb{C}}
\newcommand{\Sbb}{\mathbb{S}}

\newcommand{\card}[1]{\lvert #1 \rvert}
\newcommand{\norm}[1]{\lVert #1 \rVert}

\DeclareMathOperator\Aut{Aut}
\DeclareMathOperator\Qut{Qut}
\DeclareMathOperator\QIso{QIso}
\DeclareMathOperator\id{id}

\DeclareFontFamily{U}{mathb}{\hyphenchar\font45}
\DeclareFontShape{U}{mathb}{m}{n}{
<5> <6> <7> <8> <9> <10> gen * mathb
<10.95> mathb10 <12> <14.4> <17.28> <20.74> <24.88> mathb12
}{}
\DeclareSymbolFont{mathb}{U}{mathb}{m}{n}
\DeclareMathSymbol{\bigastglyph}{2}{mathb}{"06}

\DeclareMathOperator*{\bigast}{\bigastglyph}

\makeatletter
\def\oversortoftilde#1{\mathop{\vbox{\m@th\ialign{##\crcr\noalign{\kern3\p@}%
      \sortoftildefill\crcr\noalign{\kern3\p@\nointerlineskip}%
      $\hfil\displaystyle{#1}\hfil$\crcr}}}\limits}
\def\sortoftildefill{$\m@th \setbox\z@\hbox{$\braceld$}%
  \braceld\leaders\vrule \@height\ht\z@ \@depth\z@\hfill\braceru$}
\makeatother

\title{Quantum Sabidussi's Theorem}
\author{Arnbjörg Soffía Árnadóttir \and Josse van Dobben de Bruyn \and Prem Nigam Kar \and David E.\ Roberson \and Peter Zeman}

\address{Department of Applied Mathematics and Computer Science, Technical University of Denmark, 2800 Kongens Lyngby, Denmark.}
\address{\noindent e-mail: {\tt sofar@dtu.dk, jdob@dtu.dk, pkar@dtu.dk, dero@dtu.dk, pezem@dtu.dk}}
\address{QMATH, Department of Mathematical Sciences, University of Copenhagen, Universitetsparken 5, 2100 Copenhagen Ø, Denmark.}
\address{\noindent e-mail: {\tt dero@dtu.dk}}
\address{\normalfont{All authors are supported by Carlsberg Semper Ardens Accelerate CF21-0682 Quantum Graph Theory.}}


\begin{document}

\begin{abstract}
	Sabidussi's theorem~[Duke Math.~J. \textbf{28}, 1961] gives necessary and sufficient conditions under which the automorphism group of a lexicographic product of two graphs is a wreath product of the respective automorphism groups.
	We prove a quantum version of Sabidussi's theorem for finite graphs, with the automorphism groups replaced by quantum automorphism groups and the wreath product replaced by the free wreath product of quantum groups.
	This extends the result of Chassaniol~[J.~Algebra \textbf{456}, 2016], who proved it for regular graphs.
	Moreover, we apply our result to lexicographic products of quantum vertex transitive graphs, determining their quantum automorphism groups even when Sabidussi's conditions do not apply.
\end{abstract}

\maketitle

\section{Introduction}

Automorphism groups of graphs play an important role when studying the interplay between group theory and graph theory. A natural question that arises in this context concerns a connection between graph products and group products: When can the automorphism group of a product of graphs be expressed as some product of their respective automorphism groups?

In this paper, we study a quantum analogue of the automorphism group. The quantum automorphism group of a finite graph $X$, was defined in 2003 by Bichon \cite{Bichon-qut-2003} as a quotient of the quantum symmetric group, $\Sbb_n^+$, introduced by Wang \cite{Wang-Qsymmetry-1998}. We shall be working with the definition given by Banica in \cite{Banica-qut-2005} and we denote the quantum automorphism group of $X$ by $\Qut(X)$.
The theory of quantum automorphism groups of graphs is still a relatively young field and we refer to \cite{Schmidt-dissertation} for a survey of the state of the art. Our goal is to look at when we can factor the quantum automorphism group of a product of graphs into a product of quantum automorphism groups. This gives us a way to construct large quantum automorphism groups from smaller ones, as well as furthering our understanding of quantum automorphism groups of graphs in general.

We will focus on the \emph{lexicographic product} of graphs (sometimes called the \emph{composition} of graphs). Informally, we can think of the lexicographic product $X[Y]$ of graphs $X$ and $Y$ as follows: replace each vertex of $X$ with a copy of $Y$ and if there is an edge between two vertices in $X$, we add a complete bipartite graph between the two corresponding copies of $Y$. It is not too hard to see that the automorphism group, $\Aut(X[Y])$, of the lexicographic product contains the wreath product, $\Aut(Y)\wr\Aut(X).$ In two papers from 1959 and 1961 \cite{sabidussi1959, sabidussi1961}, Sabidussi gives necessary and sufficient conditions for equality to hold.

Quantum automorphism groups of lexicographic products of graphs have been studied on multiple occasions in the literature. Banica and Bichon studied the quantum automorphism group of the lexicographic products of graphs under some spectral conditions in \cite{banica_quantum_2007}. They also studied the quantum automorphism group of what they call the ``free product'' of graphs, which is a colouring of the lexicographic product in \cite{Banica-Bichon-free-product}. In 2016, Chassaniol proved a quantum version of Sabidussi's theorem for finite regular graphs \cite{chassaniol2016}. Specifically, he shows that under Sabidussi's conditions, the quantum automorphism group of the lexicographic product $X[Y]$ can be written as $\Qut(Y)\wr_*\Qut(X)$,
where the $\wr_*$ denotes the free wreath product, an analogue of the wreath product for quantum groups. Here, we generalize this result to all finite graphs, dropping the regularity condition. Our main theorem is the following.

\begin{thm}
	\label{main-result}
	Let $X$ and $Y$ be two graphs
	and $X[Y]$ be their lexicographic product. Then, $\Qut(X[Y]) = \Qut(Y) \wr_* \Qut(X)$ if and only if the following conditions hold:
	\begin{enumerate}[label = \roman*.]
		\item if $Y$ is not connected, then $X$ has no twins,
		\item if $\overline{Y}$ is not connected, then $\overline{X}$ has no twins.
	\end{enumerate}
\end{thm}

The paper is organized as follows.
In \autoref{sec:Weis-Lem} we introduce the Weisfeiler--Leman algorithm \cite{WL}, which will be our main tool. \autoref{sec:main} contains the proof of our main result as well as an alternative proof to Sabidussi's original theorem. In \autoref{sec:disjoint} we give a complete characterization of quantum automorphism groups of disjoint unions of graphs and in \autoref{sec:vxtransitive} we apply the results from the previous sections to lexicographic products of quantum vertex-transitive graphs.

\section{Preliminaries}
\label{sec:prelim}

Throughout this article, we assume that all graphs are finite and simple, but not necessarily connected.
For a positive integer $n$, write $[n] = \{1,\ldots,n\}$.

\subsection{Graphs and lexicographic products}
Let $X$ be a graph. We denote by $V(X), E(X), A_X$ and $\Aut(X)$ its vertex set, edge set, adjacency matrix and automorphism group, respectively.
The \emph{neighbourhood $N_X(x)$} of a vertex $x$ is the set of neighbours of $x$.
We say that two vertices, $x$ and $y$ in a graph $X$ are \emph{twins} if $N_X(x) = N_X(y)$.
Note that, by this definition, adjacent vertices can never be twins. We denote by $\overline{X}$ the \emph{complement} of $X$, that is, the graph with vertex set $V(X)$ and edge set $\{xy: xy\notin E(X), \ x \ne y\}.$

The \emph{lexicographic product} (or \emph{composition}) of graphs $X$ and $Y$, denoted by $X[Y]$, is the graph with vertex set $V(X)\times V(Y)$ and edge set
\[E(X[Y]) \coloneqq \{(x,y)(x',y'): xx'\in E(X)\}\cup\{(x,y)(x,y'): yy'\in E(Y)\}.\]
The lexicographic product of graphs is associative and preserves complements, that is $\overline{X[Y]} = \overline{X}[\overline{Y}].$ For a fixed $x_0 \in V(X)$, denote by $Y_{x_0}$ the induced subgraph on the vertex set $\{(x_0,y) : y\in V(Y)\}$.
Clearly, $Y_x$ is isomorphic to $Y$ for all $x\in V(X)$.
Edges of $X[Y]$ having both endpoints in the same $Y_x$ will be called \emph{inner edges}, and edges going between different $Y_x$ will be called \emph{outer edges}.

An example of a lexicographic product is given in \autoref{fig:lex_example}.
It shows the lexicographic product $C_4[K_2]$, where the inner edges are coloured blue and the outer edges are coloured orange.

\begin{figure}[t]
	\centering
	\begin{tikzpicture}[vertex/.style={circle,fill,inner sep=1.4pt},
	                    inner_edge/.style={SteelBlue1,line width=1.2pt},
	                    outer_edge/.style={Orange1,line width=.7pt},
	                    Y_part/.style={fill=gray!8}]
		\node[vertex] (v0_in) at (45:1cm) {};
		\node[vertex] (v1_in) at (135:1cm) {};
		\node[vertex] (v2_in) at (-135:1cm) {};
		\node[vertex] (v3_in) at (-45:1cm) {};
		\node[vertex] (v0_out) at (45:2cm) {};
		\node[vertex] (v1_out) at (135:2cm) {};
		\node[vertex] (v2_out) at (-135:2cm) {};
		\node[vertex] (v3_out) at (-45:2cm) {};
		\foreach \x in {0,...,3} {
			\pgfmathtruncatemacro{\xn}{mod(\x + 1, 4)}
			\draw[inner_edge] (v\x_in) -- (v\x_out);
			\draw[outer_edge] (v\x_in) -- (v\xn_in);
			\draw[outer_edge] (v\x_in) -- (v\xn_out);
			\draw[outer_edge] (v\x_out) -- (v\xn_in);
			\draw[outer_edge] (v\x_out) -- (v\xn_out);
		}
	\end{tikzpicture}
	\caption{The lexicographic product $C_4[K_2]$.}
	\label{fig:lex_example}
\end{figure}
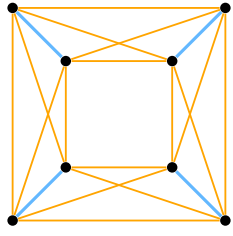

\subsection{The wreath product}
Let $H$ and $G$ be groups, let $\Omega$ be a set, and let $(g,\omega) \mapsto g\omega$ be a left action of $G$ on $\Omega$.
Write
\[ K \coloneqq \prod_{\omega \in \Omega}H_\omega, \]
where $H_\omega \cong H$ for all $\omega \in \Omega$. The \emph{wreath product} of $H$ by $G$, denoted by $H\wr G$, is the semidirect product of $K$ by $G$, where $G$ acts on $K$ by
\[ g\cdot (h_\omega)_{\omega\in\Omega} = (h_{g\omega})_{\omega\in\Omega},\qquad\text{ for $g\in G$ and $(h_\omega)_{\omega\in\Omega}\in \prod_{\omega\in\Omega}H_\omega$}. \]

\subsection{Universal \emph C*-algebras and free products}
We give a brief overview of universal $C^*$\nobreakdash-algebras and free products of $C^*$\nobreakdash-algebras in this subsection. Our treatment of this subject is mostly based on \cite[II.8.3]{blackadar}.

Let $\mathcal{G} = \{g_i\}_{i \in \Omega}$ be a set of generators and $\mathcal{R}$ a set of relations; that is, polynomials in $\{g_i\}_{i\in\Omega} \cup \{g_i^*\}_{i\in\Omega}$.
A \emph{representation} of $(\mathcal{G|R})$ is a set $\{T_i\}_{i \in \Omega}$ of bounded operators on a Hilbert space $\mathcal{H}$ satisfying the relations $\mathcal{R}$. Let $\mathcal{A}$ denote the free $*$\nobreakdash-algebra generated by $\mathcal{G}$ satisfying $\mathcal{R}$. Then, it is evident that a representation $\mathcal{(G|R)}$ defines a $*$\nobreakdash-representation of $\mathcal{A}$ on $\mathcal{H}$.

For $x \in \mathcal{A}$, we define $\norm{x} = \sup\{\norm{\pi(x)}: \pi \text{ is a representation of }\mathcal{(G|R)}\}$. If $\norm{x}$ is finite for all the generators in $\mathcal{G}$, then $\norm{\cdot}$ defines a $C^*$\nobreakdash-seminorm on $\mathcal{A}$. The completion of $\mathcal{A} / \ker(\norm{\cdot})$ with respect to $\norm{\cdot}$ is known as the \emph{universal $C^*$\nobreakdash-algebra generated by $\mathcal{G}$ subject to $\mathcal{R}$}, denoted $C^*(\mathcal{G|R})$. Every $C^*$\nobreakdash-algebra $\mathcal A$ can be expressed as a universal $C^*$\nobreakdash-algebra $C^*(\{x_a\}_{a \in A}| \mathcal{R_A})$, where:
\[ \mathcal{R_A} = \{ x_a + x_{b} = x_{a+b}, \ \lambda x_a = x_{\lambda a}, \ x_ax_{b} = x_{ab}, \ x_a^* = x_{a^*},\ \forall \ a,b \in \mathcal A, \lambda \in \C\}. \]

$C^*(\mathcal{G|R})$ satisfies the following universal property: if $\{x_i\}_{i \in \Omega}$ are elements satisfying $\mathcal R$ in a $C^*$\nobreakdash-algebra $\mathcal B$, then there is a unique $*$\nobreakdash-homomorphism $\pi: C^*(\mathcal{G|R}) \to \mathcal{B}$ such that $\pi(g_i) = x_i$ for all $i\in \Omega$.

The \emph{free product of $C^*$\nobreakdash-algebras} is the coproduct in the category of (not necessarily unital) $C^*$\nobreakdash-algebras.
In other words, if $\{\mathcal A_i\}_{i\in\Omega}$ is a collection of $C^*$\nobreakdash-algebras, then their free product is the $C^*$\nobreakdash-algebra $\bigast_{j \in \Omega} \mathcal{A}_j$ equipped with $*$\nobreakdash-homomorphisms $\varphi_i : \mathcal A_i \to \bigast_{j \in \Omega} \mathcal{A}_j$ such that the following universal property is satisfied:
For every $C^*$\nobreakdash-algebra $\mathcal B$ with $*$\nobreakdash-homomorphisms $\pi_i : \mathcal A_i \to \mathcal B$, there is a unique $*$\nobreakdash-homomorphism $\pi : \bigast_{j \in \Omega} \mathcal{A}_j \to \mathcal B$ such that $\pi_i = \pi \circ \varphi_i$ for all $i\in\Omega$; that is, the following diagram commutes:
\[ \begin{tikzcd}
	\mathcal{A}_i \arrow{r}{\varphi_i} \arrow[swap]{dr}{\pi_i} & \bigast_{i \in \Omega} \mathcal{A}_i \arrow[d, densely dashed, "\pi"] \\
	& \mathcal{B}
\end{tikzcd} \]
The free product of an arbitrary collection of $C^*$\nobreakdash-algebras always exists.
One way to see this is to construct it as follows: choose presentations $\mathcal A_i = C^*(\mathcal G_i | \mathcal R_i)$, and let $\bigast_{j \in \Omega} \mathcal A_j$ be the universal $C^*$\nobreakdash-algebra $C^*(\sqcup_{i \in \Omega} \mathcal {G}_i| \sqcup_{i \in \Omega} \mathcal{R}_i)$.
It follows easily from the universal property of universal $C^*$\nobreakdash-algebras that this satisfies the universal property of the coproduct.

If $\{\mathcal A_i\}_{i\in\Omega}$ is a collection of unital $C^*$\nobreakdash-algebras, then their coproduct in the category of unital $C^*$\nobreakdash-algebras (with unital $*$\nobreakdash-homomorphisms) also exists.
It is called the \emph{free product amalgamated over $\C$}, denoted $\bigast_{j\in\Omega}^{\C} \mathcal A_j$, and can be constructed as $\bigast_{j\in\Omega}^{\C} \mathcal A_j = \bigast_{j\in\Omega} \mathcal A_j / \langle 1_{\mathcal A_i} = 1_{\mathcal A_j} \mid i,j\in\Omega \rangle$.

The free product of two $C^*$\nobreakdash-algebras $\mathcal A$ and $\mathcal B$ is denoted $\mathcal A * \mathcal B$, and the free product amalgamated over $\C$ of two unital $C^*$\nobreakdash-algebras $\mathcal A$ and $\mathcal B$ is denoted $\mathcal A *_{\C} \mathcal B$.

The \emph{unitization} of a $C^*$\nobreakdash-algebra $\mathcal{A}$ is the $C^*$\nobreakdash-algebra
\[ \widetilde{\mathcal{A}} \coloneq \mathcal{A} \ast \mathbb{C}/ \left \langle 1 a = a 1 = a \mid a \in \mathcal{A}\right \rangle. \]
Note that if $\mathcal{A}$ is a unital $C^*$-algebra, then one has $\widetilde{\mathcal{A}} \cong \mathcal{A}$.

\subsection{Quantum automorphism groups of graphs}
\begin{definition}
	A \emph{compact quantum group} $\mathbb{G}$ is an ordered pair $(\mathcal{A},\Delta)$, where $\mathcal{A}$ is a unital $C^*$\nobreakdash-algebra and $\Delta: \mathcal{A} \to \mathcal{A} \otimes \mathcal{A}$ is a unital $*$\nobreakdash-homomorphism, known as the \emph{comultiplication} of $\mathbb{G}$, satisfying the following conditions:
	\begin{enumerate}
		\item \emph{co-associativity:} $(\Delta \otimes id) \Delta = (id \otimes \Delta) \Delta$;
		\item \emph{cancellation property:} the sets $\Delta(\mathcal{A})(1 \otimes \mathcal{A}) $ and $\Delta(\mathcal{A})(\mathcal{A} \otimes 1)$ are dense in $\mathcal{A} \otimes \mathcal{A}$.
	\end{enumerate}
\end{definition}

Let $G$ be a compact group and let $C(G)$ be the commutative $C^*$\nobreakdash-algebra of continuous complex-valued functions on $G$. Since $C(G) \otimes C(G) \cong C(G \times G)$, we may define $\Delta: C(G) \to C(G)\otimes C(G)$ by $\Delta(f)(g \times h) = f(gh)$. Then, $(C(G), \Delta)$ is a compact quantum group. Moreover, all compact quantum groups $\mathbb{G}$ where the $C^*$\nobreakdash-algebra is commutative are of this type. In view of this example, we shall denote the $C^*$\nobreakdash-algebra $\mathcal{A}$ associated with a compact quantum group $\mathbb{G} = (\mathcal{A},\Delta)$ as $C(\mathbb{G})$.

As an example, consider the \emph{quantum symmetric group} $\mathbb{S}_n^+$. Let $C(\mathbb{S}_n^+)$ denote the universal $C^*$\nobreakdash-algebra generated by $[u_{ij}]_{i,j \in [n]}$ satisfying:
\begin{align}
	\label{magic-unitary-cond}
	\begin{aligned}
		&u_{ij}^2 = u_{ij}^* = u_{ij} &&\quad\text{for all $i,j \in [n]$, and} \\[.75ex]
		&\sum_{k=1}^n u_{ik} = \sum_{k=1}^n u_{kj} = 1 &&\quad\text{for all $i,j \in [n]$}, \\
	\end{aligned}
\end{align}
and $\Delta: C(\mathbb{G}) \to C(\mathbb{G}) \otimes C(\mathbb{G})$ is the $\ast$-homomorphism defined by
\begin{equation*}
	\Delta(u_{ij}) = \sum_{k=1}^n u_{ik} \otimes u_{kj} \qquad\text{for all $i,j \in [n]$}.
\end{equation*}
A matrix satisfying \eqref{magic-unitary-cond} is known as a \emph{magic unitary}.

A compact quantum group $\mathbb{G} = (C(\mathbb{G}), \Delta)$ such that $C(\mathbb{G})$ is generated by a magic unitary $[e_{ij}]_{i,j \in [n]}$ and $\Delta$ satisfies $\Delta(e_{ij}) = \sum_{k=1}^n e_{ik} \otimes e_{kj}$ for all $i,j \in [n]$ is known as a \emph{quantum permutation group} acting on $[n]$. Let $G \subseteq S_n$ be a permutation group, and let $u_{ij}: G \to \mathbb{C}$ be the functions taking each $g \in G$ to the $(i,j)$-th component of the permutation matrix $g$. Then, $(C(G), [u_{ij}]_{i,j \in [n]})$ is a quantum permutation group. Indeed, one can show that all quantum permutation groups $\mathbb{G}$ where the $C^*$\nobreakdash-algebra $C(\mathbb{G})$ is commutative are of this form. Given a permutation group $G$, we denote the corresponding quantum permutation group as $G_q$.

Let $\mathbb{G} = (C(\mathbb{G}), u)$ and $\mathbb{H} = (C(\mathbb{H}), v)$ be two quantum permutation groups acting on $[n]$. We say $\mathbb{H}$ is a \emph{\textup(quantum\textup) subgroup} of $\mathbb{G}$, written $\mathbb{H}\subseteq \mathbb{G}$, if $u_{ij} \mapsto v_{ij}$ defines a surjective $*$\nobreakdash-homomorphism $C(\mathbb{G}) \to C(\mathbb{H})$. It is now immediately obvious that all quantum permutation groups acting on $[n]$ are subgroups of $\mathbb{S}_n^+$.

\begin{definition}
	The \emph{quantum automorphism group $\Qut(X)$} of a (possibly coloured) graph $X$ is the quantum permutation group $(C(\Qut(X)), u)$, where $C(\Qut(X))$ is the universal $C^*$\nobreakdash-algebra with generators $[v_{ij}]_{i,j \in V(X)}$ and relations
	\begin{align}
		\label{eq:qaut_def}
		u_{ij}u_{k\ell} = 0 \text{ if } ik\in E(X), j\ell\notin E(X) \text{ or vice versa},
	\end{align}
	and
	\begin{align}
		\label{eq:cqaut_def}
		u_{ij} = 0 \text{ if } c(i) \neq c(j)
	\end{align}
	if $X$ is coloured.
	The first condition is equivalent to $A_X u = u A_X $.
\end{definition}

We now introduce two binary operations on quantum permutation groups that we need.

\begin{definition}
	Let $\mathbb{G} = (C(\mathbb{G}), u)$ and $\mathbb{H} = (C(\mathbb{H}), v)$ be quantum permutation groups. Then, their \emph{free product} $\mathbb{G} \ast \mathbb{H}$ is defined as the quantum permutation group $(C(\mathbb{G}) \ast_{\mathbb{C}} C(\mathbb{H}), u \oplus v)$.
\end{definition}

\begin{definition}[{\cite{Bichon-free-wreath-product}}]
	\label{def:free-wreath-product}
	Let $\mathbb{G}=(C(\mathbb{G}), [u_{ij}]_{i,j \in [m]})$ and $\mathbb{H}=(C(\mathbb{H}), [v_{ab}]_{a,b \in [n]})$ be two quantum permutation groups. The \emph{free wreath product of $\mathbb{G}$ and $\mathbb{H}$}, denoted by $\mathbb{G} \wr_* \mathbb{H}$, is the quantum permutation group
	\[ C(\mathbb{G} \wr_* \mathbb{H}) \coloneqq (C(\mathbb{G}^{*n}) \ast_\mathbb{C} C(\mathbb{H}))/ \langle [u_{ij}^{(a)}, v_{ab}] = 0 \mid i,j \in [m], \ a,b \in [n] \rangle , \]
	where $u^{(a)}$ denotes the $a$-th diagonal block of the magic unitary $u^{(1)} \oplus \cdots \oplus u^{(n)}$ of $\mathbb{G}^{*n}$, with fundamental representation $[w_{(a,i)(b,j)}]_{(a,i),(b,j) \in [n] \times [m]}$ (of $\mathbb{G} \wr_* \mathbb{H}$) given by
	\[ w_{(a,i)(b,j)} \coloneqq u_{ij}^{(a)}v_{ab}. \]
\end{definition}
We can view the free wreath product as a quantum analogue of the wreath product of groups. Indeed, if $G, H \leq S_n$ are two permutation groups, then one can show that $G_q \wr_* G_q = (G \wr H)_q$.

Let $G$ be a group acting on a set $X$. Recall that elements $x,y \in X$ are said to be in the same \emph{orbit} of $G$ if there is a $g \in G$ such that $gx = y$. This defines an equivalence relation on $X$. Similarly, $(x,y)$ and $(x',y')$ are said to be in the same \emph{orbital} of $G$ if there exists $g \in G$ such that $gx = x'$ and $gy = y'$.

Similarly, given a quantum permutation group $\mathbb{G} = (C(\mathbb{G}), u)$ we may define a relation $\sim_1$ on $[n]$ by letting $i \sim_1 j$ if and only if $u_{ij} \neq 0$, and a relation $\sim_2$ on $[n] \times [n]$ by letting $(i,j) \sim_2 (k,\ell)$ if and only if $u_{ik}u_{j\ell} \neq 0$. It was shown in \cite{lupini-2020} that both $\sim_1$ and $\sim_2$ are equivalence relations. The equivalence classes of $[n]$ with respect to $\sim_1$ are known as \emph{orbits} of $\mathbb{G}$ and the equivalence classes of $[n] \times [n]$ with respect to $\sim_2$ are known as the \emph{orbitals} of $\mathbb{G}.$

Another closely related topic to quantum automorphism groups of graphs is quantum isomorphisms of graphs. The original formulation is based on the existence of a perfect quantum strategy for the graph isomorphism game introduced in \cite{atserias_quantum_2019}. We shall be working with an equivalent definition introduced in \cite{lupini-2020}. We shall not dive too deep into quantum isomorphisms of graphs and only introduce some definitions and results we need. We refer the reader to \cite{atserias_quantum_2019, lupini-2020} for a better understanding of these topics.

Two graphs $X$ and $Y$ are said to be \emph{quantum isomorphic}, written as $X \cong_q Y$, if there is a nonzero unital $C^*$\nobreakdash-algebra $\mathcal{A}$ and a magic unitary $u = [u_{xy}]_{x \in V(X), y \in V(Y)}$ with entries from $\mathcal{A}$ such that $A_Xu = uA_Y$.

In view of this definition, for two graphs $X$ and $Y$ with $\card{V(X)} = \card{V(Y)}$, the \emph{$C^*$\nobreakdash-algebra of the $X,Y$-isomorphism game} (as defined in \cite{paulsen_bisynchronous_2021}), denoted by $\QIso(X,Y)$ is defined as the universal $C^*$\nobreakdash-algebra generated by $[e_{xy}]_{x \in V(X), y \in V(Y)}$ satisfying:
\begin{align*}
	&e_{xy}^2 = e_{xy}^* = e_{xy} &&\quad \text{ for all } x \in V(X), y \in V(Y) \\[.75ex]
	&\sum_{y \in V(Y)}e_{xy} = 1 &&\quad \text{ for all } x \in V(X) \\
	&\sum_{x \in V(X)}e_{xy} = 1 &&\quad \text{ for all } y \in V(Y) \\[.5ex]
	&e_{xy}e_{x'y'} = 0 &&\quad \text{ if } xx' \in E(X) \text{ and } yy' \notin E(Y) \text{ or vice versa.}
\end{align*}
Given that $[e_{xy}]_{x \in V(X), y \in V(Y)}$ is a magic unitary, the last condition is equivalent to $A_X[e_{xy}]_{x \in V(X), y \in V(Y)} = [e_{xy}]_{x \in V(X), y \in V(Y)}A_Y$.

It is now not too difficult to see that for every graph $X$, one has $\QIso(X,X) = C(\Qut(X))$. In \cite{brannan_bigalois_2020}, it was proved that $X \cong_{q} Y$ if and only if $\QIso{(X,Y)}$ is not trivial. In fact, it follows from the results of \cite{brannan_bigalois_2020} that the non-triviality of the free $*$\nobreakdash-algebra generated by $[e_{xy}]_{x \in V(X), y \in V(Y)}$ satisfying above conditions is necessary and sufficient for these purposes. Hence, to show that $X \cong_{q} Y$, it suffices to show that $e_{x,y} \neq 0$ for some $x \in V(X), y \in V(Y)$.

\section{Distinguishing inner and outer edges}
\label{sec:Weis-Lem}

The main combinatorial tool in our proof is the Weisfeiler--Leman algorithm. In this section, we introduce the Weisfeiler--Leman algorithm and show that it colours inner edges (resp. inner non-edges) and outer edges (resp. outer non-edges) of the lexicographic product $X[Y]$ in different colours, provided that $X$ and $Y$ meet the conditions of Sabidussi's theorem.

\subsection{The Weisfeiler--Leman algorithm.}
We start by describing the classical (2\nobreakdash-dimensional) Weisfeiler--Leman algorithm.
For our purposes, a \emph{colouring} of a graph $X$ is any function from $V(X) \times V(X)$ to some finite set $C$.

Given a graph $X$, the Weisfeiler--Leman algorithm \cite{WL} constructs a colouring $\overline{c} : V(X) \times V(X) \to C$ in the following way.
First, construct the initial colouring
\[ c_0(x, y) \coloneqq
	\begin{cases}
		0 \quad \text{if $x = y$,}\\
		1 \quad \text{if $xy \in E(X)$,}\\
		2 \quad \text{if $xy \in E(\overline{X})$.}
	\end{cases} \]
Then, we repeat the following procedure.
Given a colouring $c : V(X) \times V(X) \to C$, for every pair of colours $i,j \in C$ we define
\[ \Delta_{ij}(x,y) \coloneqq |\{z \in V(X) : c(x, z) = i \text{ and } c(z, y) = j\}|. \]
Then the \emph{refinement} of $c$ is the colouring $c'$ given by
\[ c'(x,y) \coloneqq (c(x,y), (\Delta_{ij}(x,y))_{i,j\in C}). \]
By repeatedly refining the initial colouring $c_0$, after a finite number of steps we end up with a colouring $\overline{c}$ that is \emph{stable} in the sense that $\overline{c}$ and $\overline{c}'$ induce the same partition of $V(X) \times V(X)$.
Once we have reached this stable colouring, the algorithm returns this colouring and terminates.
The stable colouring $\overline{c}$ returned by the Weisfeiler--Leman algorithm will be called the \emph{stable colouring of $X$}.

\begin{definition}
	Let $\overline{c}$ be the stable colouring of $X$.
	\begin{enumerate}[label=(\alph*)]
		\item For vertex pairs $(x,y),(x',y') \in V(X)$, we say that \emph{$(x,y)$ and $(x',y')$ are distinguished by Weisfeiler--Leman} if and only if $\overline{c}(x,y) \neq \overline{c}(x',y')$.
		\item We say that two edges $pq,p'q' \in E(X)$ or non-edges $pq,p'q' \in E(\overline{X})$ \emph{are distinguished by Weisfeiler--Leman} if and only if
		\[ \{ \overline{c}(p,q) , \overline{c}(q,p) \} \neq \{\overline{c}(p',q') , \overline{c}(q',p') \}, \]
		or in other words, if there is no orientation of these two edges such that both the forward and backward colours match.
		\item We say that two edges $pq,p'q' \in E(X)$ or non-edges $pq,p'q' \in E(\overline{X})$ \emph{are strongly distinguished by Weisfeiler--Leman}\footnote{Non-standard terminology introduced by the authors.} if and only if
		\[ \{ \overline{c}(p,q) , \overline{c}(q,p) \} \cap \{\overline{c}(p',q') , \overline{c}(q',p') \} = \varnothing, \]
		or in other words, if for all possible orientations of these two edges, the colours differ.
	\end{enumerate}
\end{definition}

In~\cite{lupini-2020}, it was shown that the partition of $V(X) \times V(X)$ found by the 2-dimensional Weisfeiler--Leman algorithm is always a (possibly trivial) coarse-graining of the orbitals of $\Qut(X)$. This is formalised in the following result:

\begin{lem}[\cite{lupini-2020}]
	\label{lem:weis-lem-orbital}
	If $(x,y), (x',y') \in V(X) \times V(X)$ are distinguished by the Weisfeiler--Leman algorithm, then $u_{xx'}u_{yy'} = 0$.
\end{lem}

\begin{lem}
	\label{lem:longer-paths}
	Let $c : V(X) \times V(X) \to C$ be the \textup(intermediate\textup) colouring after some number of iterations of the Weisfeiler--Leman algorithm.
	Let $(x,x'),(y,y') \in V(X)^2$ be two vertex pairs and let $\ell \geq 1$ and $\vec{b} \in C^\ell$ be such that the sets
	\begin{align*}
		\{(x_0,\ldots,x_\ell) \in V(X)^{\ell+1} &\mid x_0 = x\ \text{and }\ x_\ell = x'\ \text{and }\ c(x_{i-1},x_i) = b_i\ \text{for all}\ i \in [\ell]\}, \\
		\{(y_0,\ldots,y_\ell) \in V(X)^{\ell+1} &\mid y_0 = y\ \text{and }\ y_\ell = y'\ \text{and }\ c(y_{i-1},y_i) = b_i\ \text{for all}\ i \in [\ell]\}
	\end{align*}
	have different sizes. Then $(x,x')$ and $(y,y')$ are distinguished by Weisfeiler--Leman.
\end{lem}

In other words, instead of comparing $(x,x')$ and $(y,y')$ based on the number of walks $x \rightsquigarrow x'$ (resp.{} $y \rightsquigarrow y'$) of length $2$ with specified colour profiles (as in the algorithm), we can also compare based on the number of walks of arbitrary length with specified colour profiles.
Furthermore, this comparison can be made at any intermediate step of the algorithm.
The straightforward proof of \autoref{lem:longer-paths} (by induction) is omitted.

\subsection{Weisfeiler--Leman on lexicographic products}

We proceed to execute the Weisfeiler--Leman algorithm on the lexicographic product $X[Y]$.
We show that, if $X$ and $Y$ satisfy Sabidussi's conditions (i.e.{} $Y$ is connected or $X$ has no twins, and $\overline{Y}$ is connected or $\overline{X}$ has no twins), then Weisfeiler--Leman distinguishes inner (non-)edges from outer (non-)edges in $X[Y]$.

The first iteration of the Weisfeiler--Leman algorithm already gets us a long way towards our goal, as we shall see.
For this first step, we count the paths of certain colour types explicitly, leading to the counts in \autoref{table-1}.

\begin{prop}
	\label{prop:first-count}
	Apply the Weisfeiler--Leman algorithm to $X[Y]$.
	Then, in the first iteration, for every ordered pair of vertices $(p,q) \in V(X[Y])^2$ with $pq \in E(X[Y])$, the numbers $\Delta_{ij}(p,q)$ for $i,j \in \{1,2\}$ are as depicted in \autoref{table-1}, where $n$ denotes the number of vertices of $Y$.
\end{prop}
\begin{table}[h]
	\centering
	\def\myscale{0.28}
	\caption{Triangle counts in the first iteration of Weisfeiler--Leman on $X[Y]$, where $n$ denotes the number of vertices of $Y$. In the second column, green and red arrows denote edges and non-edges of $X[Y]$, respectively. The counts $\Delta_{0j}(p,q)$ and $\Delta_{i0}(p,q)$, corresponding to degenerate triangles ($ppq$ or $pqq$), are omitted from the table.}
	\footnotesize
	\hspace*{-.3cm}
	\begin{tabular}{cccc}
		\toprule
		& & Case I: & Case II: \\
		value & \hspace*{-1cm} illustration \hspace*{-1cm} & $pq$ is an inner edge; & $pq$ is an outer edge; \\
		& & $p = (x,p_y)$, $q = (x,q_y)$ & $p = (p_x',p_y')$, $q = (q_x',q_y')$ \\
		\midrule
		$\Delta_{11}(p,q)$ & \begin{tikzpicture}[baseline = 5, everp_y node/.style={draw}, scale = \myscale]
			\useasboundingbox (-1.5,-1.2) rectangle (1.5,3);
			\draw node[circle, fill, scale = 0.3, label = below:$p$] at (-1,0) (b) {};
			\draw node[circle, fill, scale = 0.3, label= below:$q$] at (1,0) (c) {};
			\draw node[circle, fill, scale = 0.3, label = above:$r$] at (0,1.732) (a) {};
			\draw[->, line width = 0.6pt, codegreen] (b) edge (c);
			\draw[->, line width = 0.6pt, codegreen] (b) edge (a);
			\draw[->, line width = 0.6pt, codegreen] (a) edge (c);
		\end{tikzpicture}
		& $\lvert N_{Y}(p_y) \cap N_{Y}(q_y)\rvert + \lvert N_{X}(x)\rvert n $ & $\lvert N_{Y}(p_y')\rvert + \lvert N_{Y}(q_y')\rvert + \lvert N_{X}(p_x') \cap N_{X}(q_x') \rvert n$ \\
		$\Delta_{12}(p,q)$ & \begin{tikzpicture}[baseline = 5, everp_y node/.style={draw}, scale = \myscale]
			\useasboundingbox (-1.5,-1.2) rectangle (1.5,3);
			\draw node[circle, fill, scale = 0.3, label = below:$p$] at (-1,0) (b) {};
			\draw node[circle, fill, scale = 0.3, label = below:$q$] at (1,0) (c) {};
			\draw node[circle, fill, scale = 0.3, label = above:$r$] at (0,1.732) (a) {};
			\draw[->, line width = 0.6pt, codegreen] (b) edge (c);
			\draw[->, line width = 0.6pt, codegreen] (b) edge (a);
			\draw[->, line width = 0.6pt, red] (a) edge (c);
		\end{tikzpicture}
		& $\lvert N_{Y}(p_y) \cap N_{\overline{Y}}(q_y)\rvert$ & $\lvert N_{\overline{Y}}(q_y')\rvert + \lvert N_{X}(p_x') \cap N_{\overline{X}}(q_x')\rvert n$ \\
		$\Delta_{21}(p,q)$ & \begin{tikzpicture}[baseline = 5, everp_y node/.style={draw}, scale = \myscale]
			\useasboundingbox (-1.5,-1.2) rectangle (1.5,3);
			\draw node[circle, fill, scale = 0.3, label = below:$p$] at (-1,0) (b) {};
			\draw node[circle, fill, scale = 0.3, label = below:$q$] at (1,0) (c) {};
			\draw node[circle, fill, scale = 0.3, label = above:$r$] at (0,1.732) (a) {};
			\draw[->, line width = 0.6pt, codegreen] (b) edge (c);
			\draw[->, line width = 0.6pt, red] (b) edge (a);
			\draw[->, line width = 0.6pt, codegreen] (a) edge (c);
		\end{tikzpicture}
		& $\lvert N_{\overline{Y}}(p_y) \cap N_{Y}(q_y)\rvert$ & $\lvert N_{\overline{Y}}(p_y')\rvert + \lvert N_{\overline{X}}(p_x') \cap N_{X}(q_x')\rvert n$ \\
		$\Delta_{22}(p,q)$ & \begin{tikzpicture}[baseline = 5, everp_y node/.style={draw}, scale = \myscale]
			\useasboundingbox (-1.5,-1.2) rectangle (1.5,3);
			\draw node[circle, fill, scale = 0.3, label = below:$p$] at (-1,0) (b) {};
			\draw node[circle, fill, scale = 0.3, label = below:$q$] at (1,0) (c) {};
			\draw node[circle, fill, scale = 0.3, label = above:$r$] at (0,1.732) (a) {};
			\draw[->, line width = 0.6pt, codegreen] (b) edge (c);
			\draw[->, line width = 0.6pt, red] (b) edge (a);
			\draw[->, line width = 0.6pt, red
			] (a) edge (c);
		\end{tikzpicture}
		& $\lvert N_{\overline{Y}}(p_y) \cap N_{\overline{Y}}(q_y)\rvert + \lvert N_{\overline{X}}(x)\rvert n$ & $\lvert N_{\overline{X}}(p_x') \cap N_{\overline{X}}(q_x')\rvert n$ \\
		\bottomrule
	\end{tabular}
	\label{table-1}
\end{table}
\begin{proof}
	If $pq$ is an inner edge, then there are two types of proper triangles: either the third vertex $r$ belongs to the same copy of $Y$ that contains both $p$ and $q$, or $r$ belongs to a different copy of $Y$.
	\begin{center}
		\def\noderadius{.3}
		\def\Yfactorradius{.7}
		\def\labeldist{-1.5}
		\begin{tikzpicture}[scale=.9]
			\begin{scope}
				\draw (0,0) circle (\Yfactorradius);
				\draw (-120:\noderadius) node[vertex] (p1) {} ++(180:6pt) node {$p$};
				\draw (120:\noderadius) node[vertex] (q1) {} ++(180:6pt) node {$q$};
				\draw (0:\noderadius) node[vertex] (r1) {} ++(0:6pt) node {$r$};
				\draw[->] (p1) -- (q1);
				\draw[->] (p1) -- (r1);
				\draw[->] (r1) -- (q1);
				\node[anchor=base] at (0,\labeldist) {Case 1A: $r_x = x$.};
			\end{scope}
			\begin{scope}[xshift=4.5cm]
				\draw (0,0) circle (\Yfactorradius);
				\draw (2,0) circle (\Yfactorradius);
				\draw (0,-\noderadius) node[vertex] (p2) {} ++(180:6pt) node {$p$};
				\draw (0,\noderadius) node[vertex] (q2) {} ++(180:6pt) node {$q$};
				\draw (2,0) node[vertex] (r2) {} ++(0:6pt) node {$r$};
				\draw[->] (p2) -- (q2);
				\draw[->] (p2) -- (r2);
				\draw[->] (r2) -- (q2);
				\node[anchor=base] at (1,\labeldist) {Case 1B: $r_x \neq x$.};
			\end{scope}
		\end{tikzpicture}
	\end{center}
	If $pq$ is an outer edge, then there are three types of proper triangles: either the third vertex $r = (r_x',r_y')$ belongs to the same copy of $Y$ that contains $p$ (that is, $r_x' = p_x'$), or $r$ belongs to the same copy of $Y$ that contains $q$ (that is, $r_x' = q_x'$), or $r$ belongs to yet another, disjoint copy of $Y$ (that is, $r_x' \neq p_x',q_x'$).
	\begin{center}
		\def\noderadius{.3}
		\def\Yfactorradius{.7}
		\def\labeldist{-1.5}
		\begin{tikzpicture}[scale=.9]
			\begin{scope}[xshift=-.5cm]
				\draw (0,0) circle (\Yfactorradius);
				\draw (0,2) circle (\Yfactorradius);
				\draw (-\noderadius,0) node[vertex] (p1) {} ++(-120:7pt) node {$p$};
				\draw (0,2) node[vertex] (q1) {} ++(90:7pt) node {$q$};
				\draw (\noderadius,0) node[vertex] (r1) {} ++(-60:7pt) node {$r$};
				\draw[->] (p1) -- (q1);
				\draw[->] (p1) -- (r1);
				\draw[->] (r1) -- (q1);
				\node[anchor=base] at (0,\labeldist) {Case 2A: $r_x' = p_x'$.};
			\end{scope}
			\begin{scope}[xshift=3.5cm]
				\draw (0,0) circle (\Yfactorradius);
				\draw (0,2) circle (\Yfactorradius);
				\draw (0,0) node[vertex] (p2) {} ++(-90:7pt) node {$p$};
				\draw (-\noderadius,2) node[vertex] (q2) {} ++(120:7pt) node {$q$};
				\draw (\noderadius,2) node[vertex] (r2) {} ++(60:7pt) node {$r$};
				\draw[->] (p2) -- (q2);
				\draw[->] (p2) -- (r2);
				\draw[->] (r2) -- (q2);
				\node[anchor=base] at (0,\labeldist) {Case 2B: $r_x' = q_x'$.};
			\end{scope}
			\begin{scope}[xshift=7cm]
				\draw (0,0) circle (\Yfactorradius);
				\draw (0,2) circle (\Yfactorradius);
				\draw (30:2cm) circle (\Yfactorradius);
				\draw (0,0) node[vertex] (p3) {} ++(180:7pt) node {$p$};
				\draw (0,2) node[vertex] (q3) {} ++(180:7pt) node {$q$};
				\draw (30:2cm) node[vertex] (r3) {} ++(0:7pt) node {$r$};
				\draw[->] (p3) -- (q3);
				\draw[->] (p3) -- (r3);
				\draw[->] (r3) -- (q3);
				\node[anchor=base] at (.866,\labeldist) {Case 2C: $r_x \neq p_x,q_x$.};
			\end{scope}
		\end{tikzpicture}
	\end{center}
	\autoref{table-1} is obtained by counting the number of triangles for each of these types.
	The terms divisible by $n$ arise from triangle types where $r_x \neq p_x,q_x$ (that is, Case 1B and Case 2C), and the terms that evaluate to $0$ (and are therefore ``missing'' from the table) arise from triangle types involving both an edge and a non-edge going between the same two copies of $Y$, which does not occur in the lexicographic product (this can happen in Case 1B, Case 2A and Case 2B).
	Detailed verification of the counts in \autoref{table-1} is left to the reader.
\end{proof}

\begin{prop}
	\label{edge-tri-count}
	Apply the Weisfeiler--Leman algorithm to $X[Y]$.
	If the first iteration of the algorithm does not strongly distinguish an inner edge $pq = (x,p_y)(x,q_y)$ from an outer edge $p'q' = (p_x',p_y')(q_x',q_y')$, then:
	\begin{enumerate}[label = \roman*.]
		\item $p_x'$ and $q_x'$ are twins in $\overline{X}$, and
		\item $N_{\overline{Y}}(p_y) \cap N_{\overline{Y}}(q_y) = \varnothing$.
	\end{enumerate}
\end{prop}
\begin{proof}
	Let $c_1 : V(X[Y])^2 \to C$ denote the colouring function after the first iteration of Weisfeiler--Leman.
	Since $pq$ and $p'q'$ are not strongly distinguished by $c_1$, we may assume that $c_1(p,q) = c_1(p',q')$ (interchange $p \leftrightarrow q$ and/or $p' \leftrightarrow q'$ if necessary.)
	Then all triangle counts for the ordered pairs $(p,q)$ and $(p',q')$ are equal (cf.{} \autoref{table-1}).
	In particular, looking at the second row of \autoref{table-1}, we have
	\[ \lvert N_{Y}(p_y) \cap N_{\overline{Y}}(q_y)\rvert = \Delta_{12}(p,q) = \Delta_{12}(p',q') = \lvert N_{\overline{Y}}(q_y')\rvert + \lvert N_{X}(p_x') \cap N_{\overline{X}}(q_x')\rvert n, \]
	where $n = \card{V(Y)}$.
	Since $N_{Y}(p_y) \cap N_{\overline{Y}}(q_y) \subseteq V(Y) \setminus \{p_y,q_y\}$, we have $\Delta_{12}(p,q) \leq n - 2 < n$, so it follows that $\card{N_{X}(p_x') \cap N_{\overline{X}}(q_x')} = 0$.
	Similarly, looking at the third row of \autoref{table-1}, we find $\card{N_{\overline{X}}(p_x') \cap N_{X}(q_x')} = 0$.
	Since $p_x'q_x' \notin E(\overline{X})$, it follows that $p_x'$ and $q_x'$ are twins in $\overline{X}$.
	Moreover, looking at the last row of \autoref{table-1}, we have
	\[ \lvert N_{\overline{Y}}(p_y) \cap N_{\overline{Y}}(q_y)\rvert + \lvert N_{\overline{X}}(x)\rvert n = \Delta_{22}(p,q) = \Delta_{22}(p',q') = \lvert N_{\overline{X}}(p_x') \cap N_{\overline{X}}(q_x')\rvert n. \]
	Since $\card{N_{\overline{Y}}(p_y) \cap N_{\overline{Y}}(q_y)} \leq n - 2 < n$ and all other summands are divisible by $n$, we must have $\card{N_{\overline{Y}}(p_y) \cap N_{\overline{Y}}(q_y)} = 0$.
\end{proof}

By symmetry (replacing $X$ and $Y$ by $\overline{X}$ and $\overline{Y}$), we also get the following complementary result for non-edges:

\begin{prop}
	\label{non-edge-tri-count}
	Apply the Weisfeiler--Leman algorithm to $X[Y]$.
	If the first iteration of the algorithm does not strongly distinguish an inner non-edge $pq = (x,p_y)(x,q_y)$ from an outer non-edge $p'q' = (p_x',p_y')(q_x',q_y')$, then:
	\begin{enumerate}[label = \roman*.]
		\item $p_x'$ and $q_x'$ are twins in $X$, and
		\item $N_{Y}(p_y) \cap N_{Y}(q_y) = \varnothing$.\qed
	\end{enumerate}
\end{prop}

Now we are ready to prove the main result of this section.

\begin{thm}
	\label{weis-lem-sabi}
	Let $X$ and $Y$ be two graphs satisfying the following properties:
	\begin{enumerate}[label = \textup(\roman*\textup)]
		\item if $Y$ is not connected, then $X$ has no twins, and
		\item if $\overline{Y}$ is not connected, then $\overline{X}$ has no twins.
	\end{enumerate}
	Then the Weisfeiler--Leman algorithm, applied to $X[Y]$, strongly distinguishes all inner edges from all outer edges, and all inner non-edges from all outer non-edges.
\end{thm}
\begin{proof}
	We prove that Weisfeiler--Leman strongly distinguishes inner edges from outer edges; the proof for non-edges follows analogously.
	
	If $\overline{X}$ has no twins, then it follows from \autoref{edge-tri-count} that all inner edges are strongly distinguished from all outer edges.
	
	Assume now that $\overline{Y}$ is connected.
	Let $c_1 : V(X[Y])^2 \to C$ be the colouring function after the first iteration of the Weisfeiler--Leman algorithm, and let $F \subseteq E(X[Y])$ be the set of all inner edges $pq = (x,p_y)(x,q_y)$ such that $N_{\overline{Y}}(p_y) \cap N_{\overline{Y}}(q_y) = \varnothing$; that is, $p$ and $q$ do not have a common inner non-neighbour.
	By \autoref{edge-tri-count}, the first iteration of Weisfeiler--Leman strongly distinguishes all edges in $F$ from all outer edges, so there is a set $C_F$ of colour classes such that $c_1(p,q) , c_1(q,p) \in C_F$ for all $pq \in F$ and $c_1(p',q'),c_1(q',p') \notin C_F$ for all outer edges $p'q'$.
	
	Consider an inner edge $pq = (x,p_y)(x,q_y)$ and an outer edge $p'q' = (p_x',p_y')(q_x',q_y')$.
	Since $\overline{Y}$ is connected, we may choose a shortest path $p_y = y_0 y_1 \cdots y_\ell = q_y$ between $p_y$ and $q_y$ in $\overline{Y}$.
	Note that $y_iy_j \in E(Y)$ whenever $|i - j| > 1$, for otherwise there would be a shorter path between $p_y$ and $q_y$ in $E(\overline{Y})$.
	Now write $k = \lceil \frac{\ell}{2} \rceil$ and consider the vertex tuple $(r_0,\ldots,r_k) \in V(X[Y])^{k+1}$ given by
	\[ (r_0,r_1,\ldots,r_k) = \begin{cases}
	((x,y_0),(x,y_2),\ldots,(x,y_{2k})) & \text{if $\ell = 2k$}, \\
	((x,y_0),(x,y_2),\ldots,(x,y_{2k-2}),(x,y_{2k-1})) & \text{if $\ell = 2k - 1$}.
	\end{cases} \]
	(In other words, we define $r$ by going from $p$ to $q$ by taking steps of size $2$ and at most one step of size $1$ along the path $(x,y_0)(x,y_1)\cdots(x,y_\ell)$ in $X[Y]$.)
	If $\ell$ is even or $i < k$, then $r_{i-1}r_i = (x,y_{2i-2})(x,y_{2i})$ is an inner edge of $X[Y]$ (since $y_{2i-2}y_{2i} \in E(Y)$, by the above), and its endpoints have a common inner non-neighbour $(x,y_{2i-1})$, so we have $r_{i-1}r_i \in F$.
	It follows that there is a walk\footnote{We use the term ``walk'' a bit loosely here, since it can include both edges and non-edges. This makes sense if one adopts the Weisfeiler--Leman point of view by thinking of the graph $X[Y]$ as being a complete graph on the vertex set $V(X[Y])$ together with a colouring function $V(X[Y])^2 \to C$. From that perspective, non-edges are just edges in a different colour.} from $p$ to $q$ using a number of edges from $F$ followed by at most one non-edge.
	
	We show that a walk with the same sequence of colours does not exist between $p'$ and $q'$.
	Assume, for contradiction, that there exists a walk $(r_0',\ldots,r_k')$ of length $k$ between $r_0' = p'$ and $r_k' = q'$ such that $c_1(r_{i-1}',r_i') = c_1(r_{i-1},r_i)$ for all $i \in [k]$.
	Since $c_1$ is a refinement of $c_0$, we have $r_{i-1}'r_i' \in E(X[Y])$ if and only if $r_{i-1}r_i \in E(X[Y])$, so we see that the walk $r_0'r_1'\cdots r_k$ consists of a number of edges followed by at most one non-edge (in case $\ell$ is odd).
	Furthermore, the edges among these are coloured (by $c_1$) using colours from $C_F$, so they must be inner edges (no outer edge receives a colour from $C_F$).
	But this contradicts the structure of $X[Y]$, for the following reasons:
	\begin{itemize}
		\item If $\ell$ is even, then the walk $p' = r_0'r_1'\cdots r_k' = q'$ consists entirely of inner edges, contrary to the assumption that $p'$ and $q'$ belong to different copies of $Y$ (that is, $p'_x \neq q'_x$, because $p'q'$ is an outer edge).
		\item If $\ell$ is odd, then the walk $p' = r_0'r_1'\cdots r_k' = q'$ consists of $k - 1$ inner edges followed by a non-edge. But then it follows that $p'$ and $r_{k-1}'$ belong to the same copy of $Y$, but there is an edge $p'q'$ and no edge $r_{k-1}'q'$, contrary to the definition of the lexicographic product.
	\end{itemize}
	This situation is illustrated in \autoref{fig:weis-lem-sabi-situation}.
	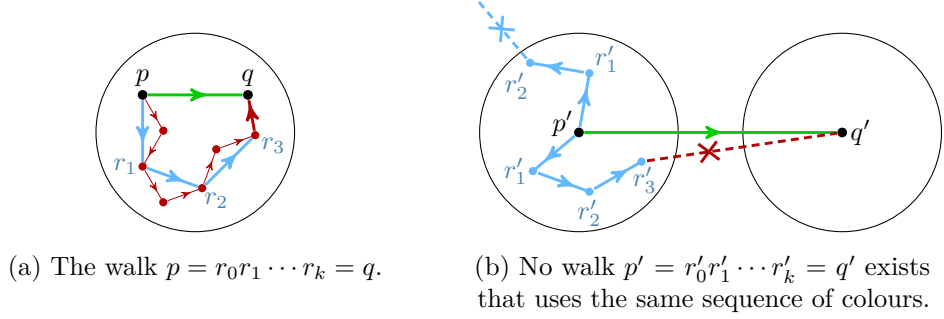
\begin{figure}[h]
		\centering
		\def\Yfactorradius{1.32}
		\def\labeldist{-1.9}
		\def\pqX{.7}
		\def\pqY{.5}
		\def\thinredarrowlength{.55}
		\def\bluearrowlength{.8}
		\def\Xsize{3.5}
		\begin{tikzpicture}[scale=1,small_vertex/.style={circle,fill,inner sep=1.1pt},
			            inner_edge/.style={green!80!black,semithick},
			            inner_non_edge/.style={red!70!black,line width=.3pt},
			            outer_non_edge/.style={red!70!black},
			            inner_blue_edge/.style={SteelBlue1,semithick},
			            inner_blue_edge_label/.style={SteelBlue1!70!black},
			            my_emphasis/.style={line width=1.1pt}]
			\begin{scope}[decoration={markings,mark=at position .5 with {\arrow[xshift=1.5pt+2.25\pgflinewidth]{>}}}] 
				\begin{scope}
					\draw (0,0) circle (\Yfactorradius);
					\begin{scope}[yshift=\pqY cm]
						\draw (-\pqX,0) node[vertex] (p) {} ++(90:6.5pt) node {$p$};
						\draw (\pqX,0) node[vertex] (q) {} ++(90:6.5pt) node {$q$};
						\draw[postaction={decorate},inner_edge,my_emphasis] (p) -- (q);
						\draw[postaction={decorate},inner_non_edge] (p) -- ++(-60:\thinredarrowlength) node[small_vertex] (r1) {};
						\draw[postaction={decorate},inner_non_edge] (r1) -- ++(-120:\thinredarrowlength) node[small_vertex] (r2) {} ++(180:6.5pt) node[inner_blue_edge_label] {$r_1$};
						\draw[postaction={decorate},inner_non_edge] (r2) -- ++(-60:\thinredarrowlength) node[small_vertex](r3) {};
						\draw[postaction={decorate},inner_non_edge] (r3) -- ++(20:\thinredarrowlength) node[small_vertex] (r4) {} ++(-40:7pt) node[inner_blue_edge_label] {$r_2$};
						\draw[postaction={decorate},inner_non_edge] (r4) -- ++(70:\thinredarrowlength) node[small_vertex] (r5) {};
						\draw[postaction={decorate},inner_non_edge] (r5) -- ++(20:\thinredarrowlength) node[small_vertex] (r6) {} ++(-30:8pt) node[inner_blue_edge_label] {$r_3$};
						\draw[postaction={decorate},inner_non_edge,my_emphasis] (r6) -- (q);
						\draw[postaction={decorate},inner_blue_edge,my_emphasis] (p) -- (r2);
						\draw[postaction={decorate},inner_blue_edge,my_emphasis] (r2) -- (r4);
						\draw[postaction={decorate},inner_blue_edge,my_emphasis] (r4) -- (r6);
					\end{scope}
					\node[anchor=base] at (0,\labeldist) {(a) The walk $p = r_0r_1\cdots r_k = q$.};
				\end{scope}
				\begin{scope}[xshift=5.1cm]
					\draw (0,0) circle (\Yfactorradius);
					\draw (3.5,0) circle (\Yfactorradius);
					\draw (0,0) node[vertex] (p) {} ++(150:7pt) node {$p'$};
					\draw (3.5,0) node[vertex] (q) {} ++(0:7pt) node {$q'$};
					\draw[postaction={decorate},inner_edge,my_emphasis] (p) -- (q);
					\draw[postaction={decorate},inner_blue_edge,my_emphasis] (p) -- ++(80:\bluearrowlength) node[small_vertex] (r1) {} ++(40:8.5pt) node[inner_blue_edge_label] {$r_1'$};
					\draw[postaction={decorate},inner_blue_edge,my_emphasis] (r1) -- ++(170:\bluearrowlength) node[small_vertex] (r2) {} ++(-120:10pt) node[inner_blue_edge_label] {$r_2'$};
					\draw[postaction={decorate},inner_blue_edge,my_emphasis,>={Rays[width=\Xsize mm,length=\Xsize mm]},densely dashed] (r2) -- ++(130:1.1cm);
					\draw[postaction={decorate},inner_blue_edge,my_emphasis] (p) -- ++(-140:\bluearrowlength) node[small_vertex] (r1) {} ++(180:6.5pt) node[inner_blue_edge_label] {$r_1'$};
					\draw[postaction={decorate},inner_blue_edge,my_emphasis] (r1) -- ++(-20:\bluearrowlength) node[small_vertex] (r2) {} ++(-90:8pt) node[inner_blue_edge_label] {$r_2'$};
					\draw[postaction={decorate},inner_blue_edge,my_emphasis] (r2) -- ++(30:\bluearrowlength) node[small_vertex] (r3) {} ++(-80:7pt) node[inner_blue_edge_label] {$r_3'$};
					\node[anchor=base,text width=6.2cm] at (1.75,\labeldist) {(b) No walk $p' = r_0'r_1'\cdots r_k' = q'$ exists that uses the same sequence of colours.};
				\end{scope}
			\end{scope}
			\draw[outer_non_edge,my_emphasis,postaction=decorate,decoration={markings,mark=at position .38 with {\arrow{Rays[width=\Xsize mm,length=\Xsize mm]}}},densely dashed] (r3) -- (q);
		\end{tikzpicture}
		\caption{The situation at the end of the proof of \autoref{weis-lem-sabi}. The edges coloured using colours from $C_F$ (depicted in blue) never cross from $Y_{p_x'}$ to a different copy of $Y$ (they are always inner edges), and there are no non-edges (depicted in red) between $Y_{p_x'}$ and $Y_{q_x'}$. Non-existent edges are indicated by a crossed-out dashed line.}
		\label{fig:weis-lem-sabi-situation}
	\end{figure}
	
	We conclude that there is no $k$-tuple $(r_0',\ldots,r_k') \in V(X[Y])^{k + 1}$ with $r_0' = p'$ and $r_k' = q'$ that receives the same sequence of colours as $(r_0,\ldots,r_k)$.
	As a consequence, it follows from \autoref{lem:longer-paths} that the ordered pair $(p,q)$ and the ordered pair $(p',q')$ are distinguished by Weisfeiler--Leman.
	Using the same argument for all other possible orientations of the edges $pq$ and $p'q'$, we see that the inner edge $pq$ and the outer edge $p'q'$ are strongly distinguished by Weisfeiler--Leman.
	The proof for non-edges follows analogously.
\end{proof}

\section{Quantum Sabidussi's theorem}
\label{sec:main}

In \cite{sabidussi1959,sabidussi1961}, Sabidussi gives necessary and sufficient conditions for the automorphism group of the lexicographic product of two graphs to be equal to the wreath product of the automorphism groups of the two graphs. In this section, we first give an alternative proof of Sabidussi's theorem and then prove the quantum analogue, \autoref{main-result}. We note here that Sabidussi proved a more general version, where the graphs are not necessarily finite.

\subsection{Alternative proof of Sabidussi's theorem}

We define an important subgroup of $\Aut(Z)$, where $Z = X[Y]$.
Each vertex of $Z$ is identified with a tuple $(x, y)$,
where $x \in V(X)$ and $y \in V(Y)$.
We define two types of automorphisms of $Z$.

For $x \in V(X)$ and $\psi \in \Aut(Y)$, we define
\[ \psi_x^*(x', y) \coloneqq
	\begin{cases}
		(x',\psi (y)) & \quad \text{if $x' = x$},\\
		(x', y) & \quad \text{if $x' \neq x$}.
	\end{cases} \]
The set
\[ H_x^* \coloneqq \{\psi_x^* : \psi \in \Aut(Y)\} \]
forms a subgroup of $\Aut(Z)$ that permutes the vertices of $Y_x$ while fixing the rest of the graph.

For $\varphi\in\Aut(X)$, we define
\[ \varphi^*(x, y) \coloneqq (\varphi(x), y). \]
The set
\[ G^* \coloneqq \{\varphi^* : \varphi \in \Aut(X)\} \]
forms a subgroup of $\Aut(Z)$ that permutes the copies of $Y$ in $Z$.

Finally, we define the group
\[ W_Z \coloneqq \langle G^*, H_x^* : x \in V(X)\rangle \leq \Aut(Z). \]
Recall that we denote by $\wr$ the wreath product of two groups. With the above notation, we have the following lemma, which is known, but we include the proof for completeness.

\begin{lem}
	\label{lem:wr_prod}
	We have $W_Z \cong \Aut(Y)\wr\Aut(X)$.
\end{lem}

\begin{proof}
	We claim that the group
	is isomorphic to $\Aut(Y)\wr\Aut(X)$.
	Note that $H_x^* \cong \Aut(Y)$ and $G^* \cong \Aut(X)$.
	Thus, it suffices to show that $W_Z$ is a wreath product of $H_x^*$ by $G^*$.
	
	Firstly, let
	\[ K^* \coloneqq \left\langle\bigcup_{x\in V(X)}H_x^*\right\rangle. \]
	The group $H_x^*$ centralizes $H_{x'}^*$ if and only if $x\neq x'$.
	Thus, $H_x^*$ is a normal subgroup of $K^*$.
	Further, the group
	\[ H_x^* \cap \left\langle\bigcup_{x'\neq x\in V(X)}H_{x'}^*\right\rangle \]
	is trivial.
	Therefore, the subgroup $K^*$ is the direct product
	\[ \prod_{x\in V(X)}H_x^*. \]
	
	Now, we show that $W_Z$ is a semidirect product of $K^*$ by $G^*$.
	We have
	\[ \varphi^*\psi_x^*\varphi^{*-1} = \psi_{\varphi x}^*, \]
	which implies that $K^*$ is a normal subgroup of $W_Z$ and that $W_Z = K^*G^*$.
	Clearly, every $\psi_x^*$ fixes the second coordinate and every $\varphi^*$ fixes the first coordinate.
	Thus, $K^*\cap G^*$ is trivial and $W_Z$ is a semidirect product of $K^*$ by $G^*$.
	
	Finally, the map $\Aut(Y)\wr\Aut(X) \to G$, defined by
	\[ (\psi_x)_{x \in V(X)}\varphi \mapsto \left (\prod_{x \in V(X)} \psi_x^*\right ) \varphi^*, \]
	is an isomorphism.
\end{proof}

\begin{thm}[Sabidussi]
	\label{Sabidussi}
	Let $X$ and $Y$ be graphs.
	Then $\Aut(X[Y])$ is isomorphic to $\Aut(Y)\wr\Aut(X)$ if and only if the following hold:
	(i) if $Y$ is disconnected, then $X$ has no twins, and, (ii) if $\overline{Y}$ is disconnected, then $\overline{X}$ has no twins.
\end{thm}

\begin{proof}
	If (i) and (ii) are satisfied, then, by \autoref{weis-lem-sabi}, there exists an automorphism-invariant colouring $c$ of the pairs of vertices of $Z = X[Y]$ which distinguishes the inner and outer edges of $Z$ and the inner and outer edges of $\overline{Z}$.
	For $x \in V(X)$, let $B_x \coloneqq \{(x, y) : y \in V(Y)\}$.
	Let $z, z' \in B_x$ and let $\varphi \in \Aut(Z)$
	We have that $\varphi(z) \in B_{x'}$ and $\varphi(z') \in B_{x''}$ implies $x' = x''$ since the colouring $c$ is automorphism-invariant.
	Thus, the sets $B_x$ form an automorphism-invariant partition of $V(Z)$.
	It follows that $\Aut(Z)$ is equal to the group $W_Z$, defined above.
	By \autoref{lem:wr_prod}, we have that $W_Z \cong \Aut(Y)\wr\Aut(X)$.
	
	For the converse, suppose that the condition (i) does not hold, i.e., $Y$ is disconnected and $X$ has twins $x$ and $x'$.
	Let $C \subsetneq B_x$ and $C'\subsetneq B_{x'}$ be subsets of vertices such that the respective induced subgraphs are isomorphic to the same connected component of $Y$.
	Then there exists an automorphism $\varphi \in \Aut(Z)$ which swaps $C$ and $C'$ and is identity on $V(Z)\setminus (C\cup C')$.
	Clearly, $\varphi \notin W_Z$.
	If the condition (ii) does not hold, then the same argument applies for $\overline{Y}$ and $\overline{X}$ since $\Aut(Z) = \Aut(\overline{Z}) = \Aut(\overline{X}[\overline{Y}])$.
\end{proof}

\subsection{Proof of quantum Sabidussi's theorem}
In this section, we prove our main result, \autoref{main-result}, which is a quantum version of Sabidussi's theorem.
First, we define the \emph{free product} of graphs as introduced in \cite{Banica-Bichon-free-product}. The \emph{free product} of two edge-coloured graphs $X$ and $Y$ with colours $c_X : E(X) \to \{1,2,\dots, p\}$ and $c_Y: E(Y) \to \{1,\dots,q\}$ is the edge-coloured graph $X \ast Y$. The following subsets of $V(X) \times V(Y)$ times itself determine the edge-coloured graph $X \ast Y$:
\begin{enumerate}
	\item $E_r^0 \coloneqq \{(x,y)(x',z): xx' \in E(X), c_X(xx') = r, y,z \in V(Y)\}$
	\item $F_s^0 \coloneqq \{(x,y)(x,z): y,z \in E(Y), c_Y(yz) = s, x \in V(X)\}$
\end{enumerate}

The following theorem was proved in \cite{Banica-Bichon-free-product}:

\begin{thm}[{\cite[Theorem 6.1]{Banica-Bichon-free-product}}]
	\label{free-prod}
	If $X$ and $Y$ are complete edge coloured graphs, then $\Qut(X \ast Y) = \Qut(Y) \wr_* \Qut(X) $.
\end{thm}

It should be noted that if $X$ and $Y$ are any graphs, we may form the complete edge-coloured graphs $\widetilde{X}$ and $\widetilde{Y}$ by colouring the edges in $X$ and $Y$ green, and colouring the non-edges red. Then, it is clear that the free product $\widetilde{X} \ast \widetilde{Y}$ is a refinement of the complete edge-coloured graph $\widetilde{X[Y]}$. It is also not too difficult to see that for any graph $X$, $\Qut(\widetilde{X}) = \Qut(X)$. Hence, for all graphs $X$ and $Y$ we have
\[ \Qut(Y) \wr_* \Qut(X) = \Qut(\widetilde{Y}) \wr_* \Qut(\widetilde{X}) = \Qut(\widetilde{X} \ast \widetilde{Y}) \subseteq \Qut(\widetilde{X[Y]}) = \Qut(X[Y]). \]
Now, we shall prove that $\Qut(X[Y]) = \Qut(\widetilde{X} \ast \widetilde{Y})$, under the assumptions of \autoref{weis-lem-sabi}. This will complete the proof of \autoref{main-result}.

\begin{lem}
	\label{lem:wre-and-free}
	Let $X$ and $Y$ be two graphs such that
	\begin{enumerate}[label = \roman*.]
		\item if $Y$ is not connected, then $X$ has no twins,
		\item if $\overline{Y}$ is not connected, then $\overline{X}$ has no twins,
	\end{enumerate}
	and $X[Y]$ be their lexicographic product. Then, $\Qut(X[Y]) = \Qut(\widetilde{X} \ast \widetilde{Y})$.
\end{lem}
\begin{proof}
	We shall show that $\Qut(\widetilde{X}[\widetilde{Y}]) = \Qut(\widetilde{X} \wr_* \widetilde{Y})$. The result then follows from the discussion preceding this lemma. Let $[u_{ij}]_{i,j}$ denote the fundamental representation of $\Qut(\widetilde{X}[\widetilde{Y}])$ and $[v_{ij}]_{i,j}$ denote the fundamental representation of $\Qut(\widetilde{X} \ast \widetilde{Y})$. It follows from the discussion preceding this lemma that $[v_{ij}]_{i,j}$ satisfies all the relations of \eqref{eq:qaut_def}, \eqref{eq:cqaut_def} for $\widetilde{X}[\widetilde{Y}]$. Hence, $u_{ij} \mapsto v_{ij}$ defines a $*$\nobreakdash-homomorphism.
	
	It follows from the definition of $\Qut(\widetilde{X} \ast \widetilde{Y})$ that if one of $ik$ and $j\ell$ is an inner edge and the other is an outer edge, then $v_{ij}v_{k\ell} = 0$. It is not immediately clear if an analogous condition holds true for $[u_{ij}]_{i,j}$. However, it follows from \autoref{weis-lem-sabi} and \autoref{lem:weis-lem-orbital} that if one of $ik$ and $j\ell$ is an inner edge and the other is an outer edge, then $(i,k)$ and $(j,\ell)$ are in different orbitals of $\Qut(\widetilde{X}[\widetilde{Y}])$, so that $u_{ij}u_{k\ell} = 0$. Proceeding similarly, one can show that $[u_{ij}]_{i,j}$ satisfies the conditions \eqref{eq:qaut_def}, \eqref{eq:cqaut_def} for $\widetilde{X} \ast \widetilde{Y}$. Hence, $v_{ij} \mapsto u_{ij}$ defines a $*$\nobreakdash-homomorphism, so $u_{ij} \mapsto v_{ij}$ defines an isomorphism and $\Qut(\widetilde{X}[\widetilde{Y}]) = \Qut(\widetilde{X} \wr_* \widetilde{Y})$.
\end{proof}

Now, we give a proof of the main result.

\begin{proof}[Proof of \autoref{main-result}]
	Let $X$ and $Y$ be any graphs. First, we note that if $\Aut(X[Y]) \neq \Aut(Y)\wr \Aut(X)$, then $\Qut(X[Y]) \neq \Qut(Y) \wr_* \Qut(X)$. Indeed, if $\Aut(X[Y]) \neq \Aut(Y)\wr \Aut(X)$ then there is an automorphism $\varphi \in \Aut(X[Y])$ that either maps an inner edge $(w, p_y)(w, q_y)$ to an outer edge $(p_x',p_y')(q_x',q_y')$, or an inner non-edge $(w, p_y)(w, q_y)$ to an outer edge $(p_x',p_y')(q_x',q_y')$ respectively. In particular, if $W = [W_{p,q}]_{p,q \in V(X[Y])}$ is the fundamental representation of $\Qut(X[Y])$, then $W_{(w,p_y)(p_x',q_x')}W_{(w, q_y)(q_x', q_y')} \neq 0$. However, if $w = [w_{p,q}]_{p,q \in V(X[Y])}$ is the fundamental representation of $\Qut(Y) \wr_* \Qut(X)$, then $w_{(w,p_y)(p_x',q_x')}w_{(w, q_y)(q_x', q_y')} = 0$, so that $\Qut(X[Y]) \neq \Qut(Y) \wr_* \Qut(X)$. Hence, the ``only if" part follows from \autoref{Sabidussi}. The ``if" part follows from \autoref{free-prod}, \autoref{lem:wre-and-free} and the discussion preceding \autoref{lem:wre-and-free}.
\end{proof}

By applying \autoref{main-result}, we can now compute the quantum automorphism groups of several graphs that were previously unknown. In particular, we can now determine the quantum automorphism group of the lexicographic products of non-regular graphs that satisfy the conditions of Sabidussi's theorem. As an example, let $K_{1,n}$ denote the star with $n+1$ vertices. It follows from the results of \cite{Dobben-Kar-Roberson-Schmidt-Zeman} that $\Qut(K_{1,n}) = \mathbb{S}_n^+$. Clearly, $K_{1,n}$ is connected for each $n \in \mathbb{N}$, and it can be verified that $\overline{K_{1,m}}$ has no twins for each $m \in \mathbb{N}$. Hence, by \autoref{main-result} we have that $\Qut(K_{1,m}[K_{1,n}]) = \mathbb{S}_n^+\wr_* \mathbb{S}_m^+$.

\section{Quantum automorphism group of disjoint union of graphs}
\label{sec:disjoint}
In this section, we give a complete characterization of the quantum automorphism group of disjoint unions of connected graphs. Some special cases are previously known:
\begin{itemize}
	\item The quantum automorphism group of a disjoint union of pairwise non-quantum isomorphic connected graphs $\{X_i\}_{i \in [n]}$ is given by the free product of the quantum automorphism groups of the individual graphs, $\bigast_{i=1}^n \Qut(X_i)$. In fact, \cite[Lemma 6.4]{Dobben-Kar-Roberson-Schmidt-Zeman} shows the stronger result that if $\{X_i\}_{i=1}^n$ are any graphs such that no connected component of $X_i$ is quantum isomorphic to a connected component of $X_j$, if $i \neq j$, then $\Qut(\bigsqcup_{i=1}^n X_i) = \bigast_{i=1}^n \Qut(X_i)$.
	\item The quantum automorphism group of a disjoint union of $n$ isomorphic copies of a connected graph $X$ is given by the free wreath product of the quantum automorphism group of $X$ by the quantum symmetric group, $\mathbb{S}_n^+$, $\Qut(X) \wr_* \mathbb{S}_n^+$. See \cite{Banica-Bichon-free-product, Bichon-free-wreath-product} for a proof.
\end{itemize}

From these cases, the disjoint union of any finite collection of connected graphs, where any two graphs are either isomorphic or not quantum isomorphic, can be calculated in terms of free products and free wreath products of the quantum automorphism groups of the individual graphs. See the proof of \cite[Theorem 1.1]{Dobben-Kar-Roberson-Schmidt-Zeman} for more details.

It should be noted that the quantum automorphism group of the disjoint union of non-isomorphic quantum isomorphic graphs cannot always be expressed in terms of the quantum automorphism groups of the individual graphs. Indeed, we may construct a pair of non-isomorphic quantum isomorphic graphs $X, Y$ such that $\card{E(Y)} \neq \card{E(\overline{X})}$, which in particular implies that $\overline{X} \not\cong_q Y$.
Then, the quantum automorphism group of $\overline{X} \sqcup Y$ is $\Qut(\overline{X}) \ast \Qut(Y) = \Qut(X) \ast \Qut(Y)$, whereas the quantum automorphism group of $X \sqcup Y$ is not $\Qut(X) \ast \Qut(Y)$, as we shall see. Hence, the characterization we give is not a ``product" of quantum groups, and it requires some additional information about the graphs. In this section, we give a complete characterization of the quantum automorphism group of the disjoint union of graphs. First, we state a lemma that we need.

\begin{lem}[{\cite[Lemma 3.2.2]{Schmidt-dissertation}}]
	\label{lem:dist-from-vert}
	Let $X$ be a graph and let $[u_{xx'}]_{x,x' \in V(X)}$ be the fundamental representation of $\Qut(X)$.
	If the distance between $x$ and $y$ is different from the distance between $x'$ and $y'$, then $u_{xx'}u_{yy'} = 0$.
\end{lem}

\begin{lem}
	\label{lem:disj-uni}
	Let $\{X_i\}_{i=1}^n$ be a family of connected graphs, let $X \coloneqq \bigsqcup_{i=1}^n X_i$, and let $[w_{pq}]_{p,q \in V(X)}$ be the fundamental representation of $\Qut(X)$.
	For $i, j \in [n]$, let $B_{ij}$ denote the submatrix of $[w_{pq}]_{p,q \in V(X)}$ whose rows and columns are indexed by $V(X_i)$ and $V(X_j)$; that is, $B_{ij} \coloneqq [w_{pq}]_{p \in V(X_i), q \in V(X_j)}$.
	Then:
	\begin{enumerate}[label = \textup(\alph*\textup)]
		\item\label{itm:disj:common-sum} For fixed $i,j \in [n]$, the rows and columns of $B_{ij}$ all sum to the same element.
		\item\label{itm:disj:magic} For $i,j \in [n]$, let $e_{ij}$ denote the common row and column sum from \ref{itm:disj:common-sum}. Then $[e_{ij}]_{i,j \in [n]}$ is a magic unitary, and $e_{ij}w_{pq} = w_{pq}$ for all $p \in V(X_i)$ and $q \in V(X_j)$.
		\item\label{itm:disj:qiso} For all $i,j \in [n]$, we have $A_{X_i} B_{ij} = B_{ij}A_{X_j}$.
	\end{enumerate}
\end{lem}
\begin{proof}
	\ref{itm:disj:common-sum} By \autoref{lem:dist-from-vert}, if $p,p' \in V(X_i)$, $q \in V(X_j)$ and $q' \notin V(X_j)$, then $w_{pq}w_{p'q'} = 0$.
	Hence, for all $p \in X_i$ and all $q' \in X_j$ we have
	\begin{align*}
		\sum_{q \in V(X_j)} w_{pq} & = \Bigg(\sum_{q \in V(X_j)}w_{pq}\Bigg)\Bigg(\sum_{p' \in V(X)} w_{p'q'}\Bigg) \\
		& = \Bigg(\sum_{q \in V(X_j)} w_{pq}\Bigg)\Bigg(\sum_{p' \in V(X_i)} w_{p'q'}\Bigg) \\
		& = \Bigg(\sum_{q \in V(X)} w_{pq}\Bigg)\Bigg(\sum_{p' \in V(X_i)} w_{p'q'}\Bigg) \\
		& = \sum_{p' \in V(X_i)} w_{p'q'}.
	\end{align*}
	This shows that every row of $B_{ij}$ sums to the same value as every column of $B_{ij}$.
	It follows that the rows and columns of $B_{ij}$ all sum to the same value.
	
	\ref{itm:disj:magic}
	Since $e_{ij}$ is the sum of a number of pairwise orthogonal self-adjoint projections, $e_{ij}$ is a self-adjoint projection as well.
	Furthermore, clearly the row and column sums of $[e_{ij}]_{i,j \in [n]}$ correspond to row and column sums of $[w_{pq}]_{p,q \in V(X)}$, which evaluate to $1$ since $[w_{pq}]_{p,q \in V(X)}$ is a magic unitary.
	This shows that $[e_{ij}]_{i,j \in [n]}$ is a magic unitary as well.
	Finally, note that, for $p \in V(X_i)$ and $q \in V(X_j)$, we have $e_{ij} w_{pq} = (\sum_{q' \in V(X_j)} w_{pq'})w_{pq} = w_{pq}$.
	
	\ref{itm:disj:qiso} The fundamental representation $[w_{pq}]_{p,q \in V(X)}$ of $\Qut(X)$ commutes with $A_X$.
	Note that the first of these two matrices is in block form $[B_{ij}]_{i,j \in [n]}$ and the second is in block diagonal form $A_X = \bigoplus_{i=1}^n A_{X_i}$.
	By comparing the $(i,j)$-th blocks of the block matrices $B A_X$ and $A_X B$, we see that $A_{X_i} B_{ij} = B_{ij} A_{X_j}$.
\end{proof}

\begin{thm}
	\label{thm:disj-uni}
	Let $\{X_i\}_{i=1}^n$ be a family of connected graphs, let $\{v_{pq}\}_{p \in V(X_i), q \in V(X_j)}$ be the generators of $\QIso(X_i, X_j)$, and let $[u_{ij}]_{i,j \in [n]}$ be the magic unitary of $\mathbb{S}_n^+$.
	Then the quantum automorphism group of the disjoint union $X \coloneqq \bigsqcup_{i=1}^n X_i$ is the quantum permutation group given by the $C^*$\nobreakdash-algebra
	\begin{align}
		\label{eq:disj_uni_expr}
		&\Big( \, \oversortoftilde{\!\!\Big(\bigast_{i,j=1}^n \QIso(X_i, X_j) \Big)\!} \, \ast_{\mathbb{C}} \: C(\mathbb{S}_n^+) \Big) \Big/ \Big\langle u_{ij} = 1_{\QIso(X_i, X_j)} \: \Big| \: i,j \in [n]\Big\rangle
	\end{align}
	and the fundamental representation given by the block matrix $[C_{ij}]_{i,j \in [n]}$, where $C_{ij} \coloneqq [v_{pq}]_{p \in V(X_i), q \in V(X_j)}$.
\end{thm}
Here, the brace over the free product $\bigast_{i,j=1}^n \QIso(X_i, X_j)$ refers to the unitization.
\begin{proof}
	Let $\mathcal A$ be the $C^*$\nobreakdash-algebra given by \eqref{eq:disj_uni_expr}, and let $W = [W_{pq}]_{p,q \in V(X)}$ denote the matrix obtained from $[C_{ij}]_{i,j \in [n]}$ by forgetting the block structure.
	Furthermore, let $[w_{pq}]_{p,q \in V(X)}$ denote the fundamental representation of $\Qut(X)$.
	
	Note that $[W_{pq}]_{p,q \in V(X)}$ is a magic unitary: the entries $W_{pq}$ are self-adjoint projections, and we have $\sum_{r \in V(X)} W_{pr} = \sum_{j=1}^n \sum_{r \in V(X_j)} W_{pr} = \sum_{j=1}^n u_{ij} = 1$ and likewise $\sum_{r \in V(X)} W_{rq} = 1$, for all $p,q \in V(X)$.
	Moreover, since $A_X$ is in block diagonal form $A_X = \bigoplus_{i=1}^n A_{X_i}$, and since we have $A_{X_i} C_{ij} = C_{ij} A_{X_j}$ for all $i,j \in [n]$, it follows that $WA_X = A_XW$.
	Therefore, by the universal property of $C(\Qut(X))$, there is a unique $*$\nobreakdash-homomorphism $\varphi : C(\Qut(X)) \to \mathcal A$ such that $\varphi(w_{pq}) = W_{pq}$ for all $p,q \in V(X)$.
	
	We now construct the inverse $\ast$-homomorphism.
	For all $i,j \in [n]$, it follows from the universal property of $\QIso(X_i, X_j)$ and \autoref{lem:disj-uni} that the map $v_{pq} \mapsto w_{pq}$ for all $p \in V(X_i), q \in V(X_j)$ defines a $*$\nobreakdash-homomorphism $\QIso(X_i, X_j) \to C(\Qut(X))$.
	By the universal property of the free product and the unitization, this extends to a unital $*$\nobreakdash-homomorphism
	\[ \psi_1 : \oversortoftilde{(\textstyle\bigast_{i,j=1}^n \QIso(X_i, X_j) )} \to C(\Qut(X)), \qquad v_{pq} \mapsto w_{pq}. \]
	Similarly, by the universal property of $C(\mathbb{S}_n^+)$ and \autoref{lem:disj-uni}, the map $u_{ij} \mapsto e_{ij}$ defines a unital $*$\nobreakdash-homomorphism $\psi_2 : C(\mathbb{S}_n^+) \to C(\Qut(X))$, where $[e_{ij}]_{i,j \in [n]}$ is the magic unitary in $C(\Qut(X))$ from \myautoref{lem:disj-uni}{itm:disj:magic}.
	Hence, by the universal property of the free product amalgamated over $\C$, there is a unique unital $*$\nobreakdash-homomorphism
	\[ \psi_3 : \oversortoftilde{(\textstyle\bigast_{i,j=1}^n \QIso(X_i, X_j) )} \ast_{\mathbb{C}} \: C(\mathbb{S}_n^+) \to C(\Qut(X)) \]
	extending $\psi_1$ and $\psi_2$.
	Finally, note that $\psi_3$ factors via $\mathcal A$, since it follows from \myautoref{lem:disj-uni}{itm:disj:magic} that $e_{ij}$ is the unit of the $C^*$\nobreakdash-subalgebra of $C(\Qut(X))$ generated by $\{w_{pq} \mid p \in V(X_i), q \in V(X_j)\}$, for all $i,j \in [n]$.
	Thus, there is a $*$\nobreakdash-homomorphism $\psi_4 : \mathcal A \to C(\Qut(X))$ such that $\psi_4(W_{pq}) = w_{pq}$ for all $p,q \in V(X)$.
	
	Since $\mathcal A$ is generated by $\{W_{pq} \mid p,q \in V(X)\}$, it follows that $\varphi \circ \psi_4 = \id_{\mathcal A}$. Likewise, we have $\psi_4 \circ \varphi = \id_{C(\Qut(X))}$, so we conclude that $\mathcal A \cong C(\Qut(X))$.
	Therefore, $\mathcal A$ equipped with the fundamental representation $W$ is a quantum permutation group that is isomorphic to $\Qut(X)$.
\end{proof}

\begin{note*}
	It is not too difficult to see that the expression in \eqref{eq:disj_uni_expr} collapses to the known results in the two previously known cases. Indeed, if all the graphs $X_i$ are pairwise not quantum isomorphic, then if $i \neq j$ we have $\QIso(X_i, X_j) = 0$, and therefore $e_{ij} = 0$. It follows that $e_{ii} = 1$ for all $i \in [n]$. Moreover, we have that $\QIso(X_i, X_i) = \Qut(X_i)$. Hence, the $C^*$\nobreakdash-algebra and the fundamental representation in \autoref{thm:disj-uni} are the same as that of the free product of the quantum automorphism groups of the graphs $X_i$.
	
	Similarly, if all the graphs $X_i$ are isomorphic to a graph $X$, then each of the algebras $\QIso(X_i, X_j)$ is isomorphic to $\Qut(X)$. Using this, it is not too difficult to see that the algebra and the fundamental representation in \autoref{thm:disj-uni} are the same as that of the free wreath product $\Qut(X)\wr_* \mathbb{S}_n^+$.
\end{note*}

\section{Lexicographic product of quantum vertex transitive graphs}
\label{sec:vxtransitive}

In this section, we shall focus on lexicographic products of quantum vertex transitive graphs. Recall that a graph $X$ is said to be \emph{vertex transitive} if for every two vertices $x,x' \in V(X)$ there is an automorphism $\varphi$ of $X$ such that $\varphi(x) = x'$. In \cite{dobson_automorphism_2009}, a complete characterization of the automorphism groups of lexicographic products of vertex transitive digraphs is given. In this section, we prove quantum generalizations of this result for vertex transitive and quantum vertex transitive graphs. Our approach is based on the proofs given in \cite[Section 5.1]{Dobson_Malnic_Marusic_2022}.

If $X$ is vertex transitive, the action of $\Aut(X)$ on $V(X)$ induces a single orbit. Similarly, a graph $X$ is said to be \emph{quantum vertex transitive} if the action of $\Qut(X)$ on $V(X)$ induces a single orbit. In other words, for all ordered pairs $(x,x') \in V(X) \times V(X)$, we have $u_{xx'} \neq 0$.

Let $X$ and $Y$ be quantum vertex transitive graphs. Let us assume that $\Qut(X[Y]) \neq \Qut(Y) \wr_* \Qut(X)$. Then, by \autoref{main-result}:
\begin{itemize}
	\item either $Y$ is disconnected and $X$ has twins, or
	\item $\overline{Y}$ is disconnected and $\overline{X}$ has twins.
\end{itemize}

First, let us assume that $Y$ is vertex transitive. It is easy to see that if $Y$ is disconnected, then all its connected components are vertex transitive and isomorphic to each other. Hence, we have the following result:

\begin{lem}
	\label{vert-tran-disj}
	Let $Y$ be a vertex transitive graph. Then, there exists a connected vertex transitive graph $Y'$ and $\beta \in \mathbb{N}$ such that $Y = \overline{K_\beta}[Y']$.\qed
\end{lem}

A similar result also holds true for quantum vertex transitive graphs. However, things are slightly more complicated as we can only say that the connected components are quantum isomorphic to each other. We note that it is possible for two non-isomorphic quantum vertex transitive graphs to be quantum isomorphic. Indeed, the pair of non-isomorphic, quantum isomorphic graphs in \cite{atserias_quantum_2019} are vertex transitive, and hence quantum vertex transitive.

\begin{lem}
	\label{q-vert-tran-disj}
	Let $Y$ be a disconnected quantum vertex transitive graph, and $Y_1,\ldots, Y_\beta$ be the connected components of $Y$. Then, $Y = \bigcup_{i=1}^\beta Y_i$ and the following hold true:
	\begin{enumerate}[label = \roman*.]
		\item Each connected component of $Y$ is quantum vertex transitive,
		\item The connected components of $Y$ are pairwise quantum isomorphic.
	\end{enumerate}
\end{lem}
\begin{proof}
	Let $[w_{yy'}]_{y,y' \in V(Y)}$ be the fundamental representation of $\Qut(Y)$. Since $Y$ is quantum vertex transitive, we have that $w_{yy'} \neq 0$ for all pairs $y, y' \in V(Y) \times V(Y)$. Let $y_i \in V(Y_i), y_j \in V(Y_j)$ be two vertices of $Y$, and let $[e_{y_iy_j}]_{y_i\in V(Y_i), y_j \in V(Y_j)}$ be the generators of $\QIso(Y_i, Y_j)$, and $[u_{ij}]_{i,j \in [\beta]}$ be the fundamental representation of $\mathbb{S}_{\beta}^+$. It follows from \autoref{thm:disj-uni} that $w_{y_iy_j} = u_{ij}e_{y_i y_j}$ (note that it makes sense to multiply $u_{ij}, \ e_{y_iy_j}$ even though they are in different $C^*$\nobreakdash-algebras because we are multiplying them in a quotient of the free product of the respective $C^*$\nobreakdash-algebras), so that $e_{y_iy_j} \neq 0$. In particular, $\QIso(Y_i, Y_j) \neq 0$, so that $Y_i$ and $Y_j$ are quantum isomorphic.
	
	Now, let $y,y' \in V(Y_i)$ be two vertices of $Y$ in the same connected component. Let $[v_{yy'}]_{y, y' \in V(Y_i)}$ denote the fundamental representation of $\Qut(Y_i)$. Once again, it follows from \autoref{thm:disj-uni} that $w_{yy'} = u_{ii}v_{yy'}$, so that $v_{yy'} \neq 0$. Since $y, y'$ were arbitrary, we see that $Y_i$ is quantum vertex transitive. This finishes the proof.
\end{proof}

Let $X$ be a graph. Define an equivalence relation on $V(X)$ by $x \sim x'$ if and only if $x$ and $x'$ are twins in $X$. It is clear that each equivalence class of $V(X)$ induced by $\sim$ is an independent set. When we assume that $X$ is vertex transitive, we can additionally show that all of the equivalence classes have the same size. Indeed, let $x',x^{\prime\prime}$ be two vertices of $X$ and $X', X^{\prime\prime}$ be their respective equivalence classes induced by $\sim$. Since $X$ is vertex transitive, there exists an automorphism $\varphi$ of $X$ such that $\varphi(x') = x^{\prime\prime}$. However, we also have that $\varphi(X') \subseteq \varphi(X^{\prime\prime})$. Similarly, we have that $\varphi^{-1}(X^{\prime\prime}) \subseteq X'$, so that $\varphi(X') = X^{\prime\prime}$. Hence, all the equivalence classes have the same size, say $\alpha$.

Let $\{X_i\}_{i \in V}$ be the set of equivalence classes induced by $\sim$. Let us define a graph $X'$ with vertex set $V$ as follows: $i,j$ are adjacent in $V$ if there are vertices $x_i \in X_i$ and $x_j \in X_j$ such that $x_ix_j \in E(X)$. Then, $X'$ is well-defined and $X = X'[\overline{K_{\alpha}}]$. We can also show that $X'$ is vertex transitive. This is formalized in the following result:

\begin{lem}[{\cite[Lemma 5.1.1]{Dobson_Malnic_Marusic_2022}}]
	\label{vert-tran-twin}
	Let $X$ be a vertex transitive graph. Let $\{X_i\}_{i=1}^k$ be the equivalence classes of $V(X)$ induced by $\sim$. Then, all of the equivalence classes have the same size, say $\alpha$, and there exists a vertex transitive graph $X'$ such that $X = X'[\overline{K_{\alpha}}]$.
\end{lem}

An exact analogue of the previous lemma also holds true for quantum vertex transitive graphs. First, we shall need an intermediate result that we prove.

\begin{prop}
	\label{prop:weis-lem-twin}
	Let $X$ be a graph, and $x,x', y, y' \in V(X)$ be four vertices such that $x,x'$ are twins and $y,y'$ are not twins. Then, the first iteration of the Weisfeiler--Leman algorithm distinguishes $(x,x')$ from $(y,y')$. Similarly, if $x,x'$ are twins in $\overline{X}$ and $y,y'$ are not twins in $\overline{X}$, then the first iteration of the Weisfeiler--Leman algorithm distinguishes $(x,x')$ from $(y,y')$.
\end{prop}
\begin{proof}
	Since $y,y'$ are not twins, without loss of generality, we may assume that there is a vertex $z \in X$ such that $yz \in E(X)$ and $y'z \notin E(X)$. However, no such vertex exists for $x,x'$. Hence, the first iteration of the Weisfeiler--Leman algorithm distinguishes $(x,x')$ from $(y,y')$. The other case can be established similarly.
\end{proof}

\begin{lem}
	\label{q-vert-tran-twin}
	Let $X$ be a quantum vertex transitive graph. Let $\{X_i\}_{i=1}^k$ be the equivalence classes of $V(X)$ induced by $\sim$. Then, all of the equivalence classes have the same size, say $\alpha$, and there exists a quantum vertex transitive graph $X'$ with no twins such that $X = X'[\overline{K_{\alpha}}]$.
\end{lem}
\begin{proof}
	It follows from \autoref{prop:weis-lem-twin} that the set $\bigcup_{i=1}^k X_i \times X_i$ is a union of quantum orbitals. Let us construct a graph $X^{\prime\prime}$ with vertex set $V(X)$, and edge set $E(X^{\prime\prime}) = \{(x,x'): x \neq x', \text{\ and }x,x' \text{\ are twins}\}$. Let $[w_{x_ix_j}]_{x_i, x_j \in V(X)}$ denote the fundamental representation of $\Qut(X)$. Then, we have that $w_{x_ix_j}w_{x'_ix_{j'}} = 0$, if $x_i, x_i' \in X_i$ and $x_j \in V(X_j), x_{j'} \in V(X_{j'})$, where $j \neq j'$. Hence, $[w_{x_ix_j}]_{x_i, x_j \in V(X)} A_{X^{\prime\prime}}$ satisfies \eqref{eq:qaut_def} for $X^{\prime\prime}$.
	
	Let $[v'_{x_ix_j}]_{x_i, x_j \in V(X)}$ be the fundamental representation of $\Qut(X^{\prime\prime})$. Then, since $w_{x_ix_j} \neq 0$ for all $x_i, x_j \in V(X)$, we also have that $v'_{x_ix_j} \neq 0$, so that $X^{\prime\prime}$ is quantum vertex transitive. Since $X^{\prime\prime}$ is quantum vertex transitive, it is regular. Note that the degree of each vertex $x_i \in X_i$ in $X^{\prime\prime}$ is $\card{X_i} - 1$. Hence, for all $i,j \in [k]$, we have $\card{X_i} = \card{X_j}$. Let us set the cardinality of each of these equivalence classes to be $\alpha$. By proceeding in a similar manner as the discussion preceding \autoref{vert-tran-twin}, we may construct a graph $X'$ such that $X = X'[\overline{K_{\alpha}}]$. It is also easy to see that $X'$ has no twins.
	
	Now, we shall show that $X'$ is quantum vertex transitive. Let $[u_{ij}]_{i,j \in [\alpha]}$ denote the fundamental representation of $\mathbb{S}_{\alpha}^+$, and $[v_{pq}]_{p,q \in [k]}$ be the fundamental representation of $\Qut(X')$. Then, by \autoref{main-result}, we have that $\Qut(X) = \mathbb{S}_{\alpha}^+ \wr_* \Qut(X')$. We now identify the vertex set of $X$ by $[k] \times [\alpha]$. Let $p,q \in [k]$ be arbitrary. Since $X$ is quantum vertex transitive, note that for every $i,j \in [\alpha]$, we have that $w_{(p,i)(q,j)} = v_{pq} u_{ij} \neq 0$, so that $v_{pq} \neq 0$. Since $p,q$ were arbitrary, we see that $X'$ is quantum vertex transitive.
\end{proof}

Using \autoref{vert-tran-disj} and \autoref{vert-tran-twin}, we can give a complete characterization of the automorphism group of a lexicographic product of two vertex transitive graphs. We shall state this theorem without proof, and then prove an analogue of it for the quantum automorphism group of a lexicographic product of two vertex transitive graphs.

\begin{thm}[{\cite[Theorem 5.1.8]{Dobson_Malnic_Marusic_2022}}]
	\label{thm:aut-comp-vert-tran}
	Let $X$ and $Y$ be vertex transitive graphs. If $\Aut(X[Y]) \neq \Aut(Y) \wr \Aut(X)$, then there exist positive integers $\alpha,\beta >1$ and vertex transitive graphs $X', Y'$ such that $X = X'[\overline{K_{\alpha}}]$ and $Y = \overline{K_{\beta}}[Y']$, or $X = X'[K_{\alpha}]$ and $Y = K_{\beta}[Y']$. Moreover, if $\alpha,\beta$ are chosen to be maximal, then $\Aut(X[Y]) = (\Aut(Y') \wr S_{\alpha\beta})\wr \Aut(X')$.
\end{thm}

Similarly, using \autoref{vert-tran-disj} and \autoref{vert-tran-twin}, we obtain the following direct analogue of \autoref{thm:aut-comp-vert-tran} for the quantum automorphism groups of vertex transitive graphs:

\begin{thm}
	\label{thm:qut-comp-vert-tran}
	Let $X$ and $Y$ be vertex transitive graphs. If $\Qut(X[Y]) \neq \Qut(Y) \wr_* \Qut(X)$, then there exist positive integers $\alpha,\beta >1$ and vertex transitive graphs $X', Y'$ such that $X = X'[\overline{K_{\alpha}}]$ and $Y = \overline{K_{\beta}}[Y']$, or $X = X'[K_{\alpha}]$ and $Y = K_{\beta}[Y']$. Moreover, if $\alpha,\beta$ are chosen to be maximal, then $\Qut(X[Y]) = (\Qut(Y') \wr_* \mathbb{S}^+_{\alpha\beta})\wr_* \Qut(X')$.
\end{thm}
\begin{proof}
	It follows from \autoref{main-result} that if $\Qut(X[Y]) \neq \Qut(Y) \wr_* \Qut(X)$, then one of the following holds true:
	\begin{enumerate}
		\item $Y$ is disconnected and $X$ has twins,
		\item or $\overline{Y}$ is disconnected and $\overline{X}$ has twins.
	\end{enumerate}
	
	Since at least one of $Y, \ \overline{Y}$ is connected, exactly one of the above is true. Let us assume that the former is true. Since $Y$ is vertex transitive and disjoint, it follows from \autoref{vert-tran-disj} that we may write $Y$ as $\overline{K_{\beta}}[Y']$, where $Y'$ is a connected, vertex transitive graph such that each connected component of $Y$ is isomorphic to $Y'$ and $\beta>1$. Similarly, since $X$ is vertex transitive and has twins, we may write $X$ as $X'[\overline{K_{\alpha}}]$, where $X'$ is a vertex transitive graph and $\alpha>1$. Moreover, if we choose $\alpha$ to be maximal, then $X'$ has no twins.
	
	Hence, $X[Y] = X'[\overline{K_{\alpha}}][\overline{K_{\beta}}[Y']] = X'[\overline{K_{\alpha\beta}}[Y']]$.
	Note that $\overline{K_{\alpha\beta}}[Y']$ is disconnected, so $\overline{\overline{K_{\alpha\beta}}[Y']}$ is connected, and $X'$ has no twins. Therefore, by \autoref{main-result}, $\Qut(X'[\overline{K_{\alpha\beta}}[Y']]) = \Qut(\overline{K_{\alpha\beta}}[Y'])\wr_* \Qut(X')$. Similarly, since $Y'$ is connected, we may also show that $\Qut(\overline{K_{\alpha\beta}}[Y']) = \Qut(Y') \wr_* \Qut(\overline{K_{\alpha\beta}}) = \Qut(Y') \wr_* \mathbb{S}_{\alpha\beta}^+$. Hence, $\Qut(X[Y]) = (\Qut(Y') \wr_* \mathbb{S}^+_{\alpha\beta})\wr_* \Qut(X')$ as required.
	
	Now, let us assume that $\overline{Y}$ is disconnected and $\overline{X}$ has twins. Then, $\Qut(X[Y]) = \Qut(\overline{X[Y]}) = \Qut(\overline{X}[\overline{Y}])$, and $\overline{Y}$ and $\overline{X}$ are vertex transitive graphs. Hence, using the same arguments as in the former case, we can show that there exist positive integers $\alpha,\beta >1$ and vertex transitive graphs $\overline{X'}, \overline{Y'}$ such that $\overline{X} = \overline{X'}[\overline{K_{\alpha}}]$ and $\overline{Y} = \overline{K_{\beta}}[\overline{Y'}]$. Hence, $X = X'[K_{\alpha}]$ and $Y = K_{\beta}[Y']$.
	Hence, $\Qut(X[Y]) = \Qut(\overline{X}[\overline{Y}]) = (\Qut(\overline{Y'}) \wr_* \mathbb{S}^+_{\alpha\beta})\wr_* \Qut(\overline{X'}) = (\Qut(Y') \wr_* \mathbb{S}^+_{\alpha\beta})\wr_* \Qut(X')$.
\end{proof}

We end this section by giving an analogue of \autoref{thm:qut-comp-vert-tran} for quantum vertex transitive graphs. This can also be thought of as a quantum analogue of \autoref{thm:aut-comp-vert-tran}.

\begin{thm}
	\label{thm:qut-comp-qvert-tran}
	Let $X, Y$ be quantum vertex transitive graphs such that $\Qut(X[Y]) \neq \Qut(Y) \wr_* \Qut(X)$.
	Without loss of generality, assume that $\overline{Y}$ is connected. Then, there exist
	\begin{enumerate}
		\item connected, pairwise quantum isomorphic, quantum vertex transitive graphs $\{Y_i\}_{i \in [\beta]}$, for some $\beta > 1$, such that $Y = \bigcup_{i=1}^\beta Y_i$, and
		\item a quantum vertex transitive graph $X'$ with no twins and $\alpha>1$, such that $X = X'[\overline{K_\alpha}]$,
	\end{enumerate}
	and $\Qut(X[Y]) = \Qut(\bigsqcup_{i=1}^{\alpha} \overline{K_{\beta}}[Y_i]) \wr_* \Qut[X']$.
\end{thm}
\begin{proof}
	Since $\Qut(X[Y]) \neq \Qut(Y) \wr_* \Qut(X)$, and $\overline{Y}$ is connected, we may conclude that $Y$ is disconnected and $X$ has twins. Since $X$ is quantum vertex transitive and has twins, by \autoref{q-vert-tran-twin}, there exists a quantum vertex transitive graph $X'$ with no twins and $\alpha>1$ such that $X = X'[\overline{K_{\alpha}}]$. Similarly, since $Y$ is disconnected and quantum vertex transitive, by \autoref{q-vert-tran-disj}, there exist connected, pairwise quantum isomorphic, quantum vertex transitive graphs $\{Y_i\}_{i=1}^{\beta}$, for some $\beta > 1$, such that $Y = \bigcup_{i=1}^\beta Y_i$.
	
	Hence, $X[Y] = X'[\overline{K_{\alpha}}][\bigcup_{i=1}^\beta Y_i] = X'[\overline{K_{\alpha}}[\bigcup_{i=1}^\beta Y_i]]$. Note that $\overline{K_{\alpha}}[\bigcup_{i=1}^\beta Y_i]$ is disconnected, so its complement is connected. Moreover, $X'$ has no twins, so \autoref{main-result} gives us that $\Qut(X[Y]) = \Qut(\overline{K_{\alpha}}[\bigcup_{i=1}^\beta Y_i]) \wr_* \Qut(X')$. Further simplifying $\overline{K_{\alpha}}[\bigcup_{i=1}^\beta Y_i]$ as $\bigcup_{i=1}^\beta \overline{K_{\alpha}}[Y_i]$, we finally have that $\Qut(X[Y]) = \Qut(\bigsqcup_{i=1}^\beta \overline{K_{\alpha}}[Y_i]) \wr_* \Qut[X']$.
\end{proof}

We can prove an intermediate result of \autoref{thm:qut-comp-vert-tran} and \autoref{thm:qut-comp-qvert-tran}. We shall omit the proof as it follows by combining the proofs of \autoref{thm:qut-comp-vert-tran} and \autoref{thm:qut-comp-qvert-tran}.

\begin{prop}
	Let $X$ be a quantum vertex transitive graph, and $Y$ be a vertex transitive graph. If $\Qut(X[Y]) \neq \Qut(Y) \wr_* \Aut(X)$, then there exist positive integers $\alpha,\beta >1$ and vertex transitive graphs $X', Y'$ such that $X = X'[\overline{K_{\alpha}}]$ and $Y = \overline{K_{\beta}}[Y']$, or $X = X'[K_{\alpha}]$ and $Y = K_{\beta}[Y']$. Moreover, if $\alpha,\beta$ are chosen to be maximal, then $\Qut(X[Y]) = (\Qut(Y') \wr_* \mathbb{S}^+_{\alpha\beta}) \wr_* \Qut(X')$.\qed
\end{prop}

As an application of \autoref{thm:qut-comp-vert-tran} we can now compute the quantum automorphism groups of several new graphs. As an example, let $C_4$ denote the cycle graph on four vertices. Then, by \autoref{thm:qut-comp-vert-tran}, we have that $\Qut(C_4(\overline{C_4})) = (\mathbb{S}_2^+ \wr_* \mathbb{S}_4^+)\wr_* \mathbb{S}_2^+$.

\section{Conclusion}
In 1961, Sabidussi provided necessary and sufficient conditions for the automorphism group of a lexicographic product of two graphs to be equal to the wreath product of the two automorphism groups \cite{sabidussi1961}. In 2016, Chassaniol proved a quantum version of Sabidussi's theorem for finite, regular graphs \cite{chassaniol2016}. In this paper, we have used the Weisfeiler--Leman algorithm to generalize Chassaniol's result to all finite graphs and further to give a new proof of Sabidussi's theorem for finite graphs. Moreover, we characterize the quantum automorphism groups of quantum vertex transitive graphs in the case where Sabidussi's conditions do not apply.

\bibliographystyle{plain}
\bibliography{Bibliography.bib}

\end{document}